\documentclass[oneside, 10pt]{amsart}

\usepackage{mathtools}
 \usepackage{amsmath, amscd}
  \usepackage{amsfonts}
  \usepackage{amssymb}
\usepackage{fullpage}
\usepackage{amsthm}
\usepackage{setspace}

\newtheorem{theorem}{Theorem}[section]
\newtheorem{proposition}[theorem]{Proposition}

\theoremstyle{definition}
\newtheorem*{definition}{Definition}

\newtheorem*{example}{Example}
\newtheorem*{remark}{Remark}

\usepackage{lipsum}

\newcommand\blfootnote[1]{%
  \begingroup
  \renewcommand\thefootnote{}\footnote{#1}%
  \addtocounter{footnote}{-1}%
  \endgroup
}

\DeclareMathOperator{\rsets}{\mathrm{sets-}}

\DeclareMathOperator{\Gal}{\mathrm{Gal}}
\DeclareMathOperator{\Hom}{\mathrm{Hom}}

\DeclareMathOperator{\Aut}{\mathrm{Aut}}
\DeclareMathOperator{\Res}{\mathrm{Res}}
\DeclareMathOperator{\Qalg}{\mathbb{Q}} 

\DeclareMathOperator{\Ker}{\mathrm{Ker}}
\DeclareMathOperator{\Ver}{\mathrm{Ver}}
\newcommand*{\defeq}{\mathrel{\vcenter{\baselineskip0.5ex \lineskiplimit0pt
                     \hbox{\scriptsize.}\hbox{\scriptsize.}}}%
                     =}

\begin{document}
\title{A plectic Taniyama group}
\author{Chris Blake}

\address{ Department of Mathematics \\ 
   King's College London \\
   Strand \\
  London \\
   WC2R 2LS}
\email{chris.blake@kcl.ac.uk}

\maketitle
\begin{abstract} 
We revisit Langlands' construction of the Taniyama group to define a plectic Taniyama group for any totally real field \(F\). The construction is functorial in \(F\) and recovers Langlands' original Taniyama group when \(F = \mathbb{Q}\). We relate our construction to Nekov\'a\v{r}'s `hidden symmetries' in classical CM theory, generalising the relationship between the Taniyama group and Tate's half-transfer maps.
 \end{abstract}
\tableofcontents

\section{Introduction}

\blfootnote{2000 \emph{Mathematics Subject Classification} 11G15}
\blfootnote{This research was supported by the Engineering and Physical Sciences Research Council and the Heilbronn Institute for Mathematical Research.}

This article deals with new aspects of the theory of complex multiplication. In its most general form, the `main theorem of complex multiplication' \cite{Deligne} can be expressed as an equality between two affine group schemes over \(\mathbb{Q}\) (together with some additional data) - Deligne's motivic Galois group for the Tannakian category of potentially abelian motives over \(\mathbb{Q}\)  on the one hand and the Taniyama group constructed by Langlands \cite{Langlands} on the other. Recently Nekov\'a\v{r} and Scholl \cite{SN} have initiated the study of extra `plectic' symmetries which they conjecture should act on the cohomology of certain Shimura varieties. The main purpose of this article is to generalise Langlands' construction of the Taniyama group to the plectic setting. 
 
Throughout we let \(\overline{\mathbb{Q}}\) denote the algebraic closure of \(\mathbb{Q}\) in \(\mathbb{C}\) and \(c \in \Gamma_{\mathbb{Q}} = \Gal(\overline{\mathbb{Q}} / \mathbb{Q})  \) denote the image of complex conjugation. The Taniyama group consists of an extension
                      \[ 1 \rightarrow \mathcal{S} \rightarrow \mathcal{T} \rightarrow \Gamma_{\mathbb{Q}} \rightarrow 1 \]
of the absolute Galois group \(\Gamma_{\mathbb{Q}}\)   by the Serre group \(\mathcal{S}\), together with a continuous finite-adelic splitting \(s :  \Gamma_{\mathbb{Q}} \rightarrow  \mathcal{T}(\mathbb{A}_{\mathbb{Q}, f}) \). Here \(\mathcal{S}\) is a certain explicit pro-algebraic torus over \(\mathbb{Q}\) and \(\Gamma_{\mathbb{Q}}\) is viewed as a profinite constant group scheme over \(\mathbb{Q}\). In what follows let \(F \subset \mathbb{\overline{Q}} \) be a fixed totally real field and \(\Sigma_F = \Hom_{\Qalg}(F, \overline{\mathbb{Q}}) \) denote the set of all embeddings of \(F\) into \(\overline{\mathbb{Q}}\). In this article we ape Langlands' construction and define an extension of the form \[ 1 \rightarrow \mathcal{S}_F \rightarrow \mathcal{T}_F \rightarrow \Aut_F(F\otimes_{\mathbb{Q}} \overline{\mathbb{Q}}) \rightarrow 1 \]
  together with a continuous finite-adelic splitting \(s_F :  \Aut_F(F\otimes_{\mathbb{Q}}\overline{\mathbb{Q}})  \rightarrow  \mathcal{T}_F(\mathbb{A}_{\mathbb{Q}, f}) \).          
Here:
 \begin{enumerate}
\item   The profinite group \(\Aut_F(F\otimes_{\mathbb{Q}}\overline{\mathbb{Q}})\) is the \(F\)-plectic Galois group introduced by Nekov\'a\v{r} and Scholl. There is a canonical isomorphism between \(\Aut_F(F\otimes_{\mathbb{Q}}\overline{\mathbb{Q}})\) and the `plectic group' 
\[  \Gamma_{\mathbb{Q}} \# \Gamma_{F} \defeq \{ \alpha : \Gamma_{\mathbb{Q}} \xrightarrow{\sim} \Gamma_{\mathbb{Q}}  \text{ such that for all } g \in \Gamma_{\mathbb{Q}}, h\in \Gamma_F \text{ we have } \alpha(gh) = \alpha(g) h \} \] 
consisting of all the automorphisms of \(\Gamma_{\mathbb{Q}}\) as a right \(\Gamma_F\)-set. The natural inclusion \(\Gamma_{\mathbb{Q}} \hookrightarrow \Aut_{F}(F \otimes_{\mathbb{Q}} \overline{\mathbb{Q}})\) corresponds under this isomorphism to the inclusion \(\Gamma_{\mathbb{Q}} \hookrightarrow \Gamma_{\mathbb{Q}} \# \Gamma_{F}\)  given by the left regular action. 

\item   The pro-algebraic torus \(\mathcal{S}_F\) is just the product group \( \mathcal{S}^{\Sigma_F}\).  This plays the role of \(\mathcal{S}\) in the theory of \(F\)-plectic Hodge structures\footnote{For a totally real field \(F\), an \(F\)-plectic \(\mathbb{R}\)-Hodge structure (see \cite{SN}, \S16) is a finite-dimensional real vector space \(V\) together with a decomposition of the \(\mathbb{C}\)-vector space \(V_{\mathbb{C}}\) = \(V \otimes_{\mathbb{R}} \mathbb{C}\) of the form 
\[ V_{\mathbb{C}} = \bigoplus_{\underline{p},\underline{q} \in \mathbb{Z}^{\Sigma_F}} V^{\underline{p},\underline{q}}\]
such that \(\overline{V^{\underline{p}, \underline{q}}} = V^{\underline{q},\underline{p}}\).}. \end{enumerate} 
Thanks to (1) and (2), we call our extension the \(F\)-plectic Taniyama group. The construction is functorial in \(F\) and when \(F = \mathbb{Q}\) we recover Langlands' original Taniyama group. \\

\textbf{The classical theory.}  In order to explain our results in more detail we need to recall some more of the classical picture. This theory is explained very clearly, and in more detail, in the survey articles \cite{Fargues}, \cite{Milne} and \cite{Schappacher}. For any CM number field \(K \subset \overline{\mathbb{Q}}\) and CM type \(\Phi \subset \Sigma_K = \Hom_{\Qalg}(K, \overline{\mathbb{Q}})\), Tate wrote down half-transfer map \(F_{\Phi} : \Gamma_{\mathbb{Q}} \rightarrow \Gamma_{K}^{ab}\) which is now known to describe the action of \(\Gamma_{\mathbb{Q}}\) on the torsion points of a CM abelian variety of type \(\Phi\).  This is defined by first taking a set of coset representatives \(w_{\rho}\) for \( \Gamma_K \subset \Gamma_{\mathbb{Q}}\) such that \(w_{c \rho} = cw_{\rho}\) for all \(\rho \in \Sigma_K\)  and mapping \(g \in \Gamma_{\mathbb{Q}}\) to the element \[ F_{\Phi}(g) = \prod_{\rho \in \Phi} w_{g\rho}^{-1} g w_{\rho} \;\; \text{ mod } \Gamma_{K^{ab}} \]
of \(\Gamma_K^{ab}\), which is independent of the choice of coset representatives. There is a natural way to lift \(F_{\Phi}\) through the the Artin reciprocity map to a map \(f_{\Phi} :  \Gamma_{\mathbb{Q}} \rightarrow \mathbb{A}_{K,f}^{\times}/ K^{\times} \) called the associated Taniyama element. This is determined by requiring that for any \(g \in \Gamma_{\mathbb{Q}}\) we have \(f_{\Phi}(g)^{1+c} = \chi(g) K^{\times}  \) where \( \chi :  \Gamma_{\mathbb{Q}} \rightarrow \widehat{ \mathbb{Z}}^{\times}\)  denotes the cyclotomic character. 

As we now explain, Langlands' Taniyama group is built from a universal version of these Taniyama elements. To understand this, we need to review a few more of the details of Langlands' construction. First, for any CM number field \(E \subset \overline{\mathbb{Q}}\) the level \(E\) Serre group is the quotient of the \(\mathbb{Q}\)-rational torus \(T_E = \Res_{E/\mathbb{Q}}  \mathbb{G}_m \) with character group  \[ X^{\ast}(\mathcal{S}^E) = \{ \sum_{\sigma} n_{\sigma}[\sigma] \; : \; \text{there is some } w \in \mathbb{Z} \mbox{ such that } n_{\sigma} + n_{c\sigma} = w \text{ for all } \sigma \in \Hom_{\Qalg}(E, \overline{\mathbb{Q}}) \;   \}  \]
The identity embedding \(E \subset \overline{\mathbb{Q}}\) gives rise to a canonical cocharacter \(\mu^E\) of the torus \(T_E\) and its quotient \(\mathcal{S}^E\). The full Serre group \(\mathcal{S}\) featuring in Langlands' Taniyama group is then the projective limit
\(\mathcal{S} = \varprojlim_{E} \mathcal{S}^E \)
over all CM fields \(E \subset \overline{\mathbb{Q}}\) with respect to the algebraic norm maps. The Taniyama group is itself built from extensions of the form
                      \[ 1 \rightarrow \mathcal{S}^E \rightarrow \mathcal{T}^E \rightarrow \Gamma_{E^{ab}/ \mathbb{Q}}\rightarrow 1\] together with continuous finite-adelic splittings \(s^E :  \Gamma_{E^{ab}/ \mathbb{Q}} \rightarrow  \mathcal{T}^E(\mathbb{A}_{\mathbb{Q}, f}) \), where here \(E \subset \overline{\mathbb{Q}} \) is a CM number field which is Galois over \(\mathbb{Q}\).   At this finite level, such extensions are easily classified. Any such extension gives rise to and is completely determined by a map 
                      \[\overline{b}^E :     \Gamma_{E^{ab}/ \mathbb{Q}} \rightarrow \mathcal{S}^E(\mathbb{A}_{E,f})/\mathcal{S}^E(E) \]
and it is known precisely which such maps arise in this way (see \cite{Milne-Shih} Proposition 2.7).  

The data of Langlands' level \(E\) Taniyama extension is equivalent to the data of a certain explicit map \(f^E : \Gamma_{E^{ab}/\mathbb{Q}} \rightarrow \mathcal{S}^E(\mathbb{A}_{E,f})/\mathcal{S}^E(E) \) called the universal Taniyama element. Unsurprisingly, to define this map we need to make similar choices as we did for Tate's half-transfer maps. Let  \(W_{E/\mathbb{Q}, f}\) denote the finite-adelic version of the Weil group of the extension \(E / \mathbb{Q}\), which sits in a diagram
  \[ \begin{CD} 1 @>>> \mathbb{A}_{E,f}^{\times} / E^{\times} @ >>> W_{E/\mathbb{Q}, f} @>>> \Gamma_{E/\mathbb{Q}} @>>> 1 \\
  & & @Vart_EVV @V \varphi_f VV @V \text{ id } VV & \\
  1 @>>> \Gamma_E^{ab} @>>> \Gamma_{E^{ab} / \mathbb{Q}} @>>> \Gamma_{E/\mathbb{Q}}  @>>> 1 \end{CD} \]
and pick a lift \(\tilde{c} \in W_{E/\mathbb{Q},f}\) of complex conjugation and a set \(\{w_{\sigma}\}\) of coset representatives for \(\mathbb{A}_{E,f}^{\times}/E^{\times} \subset W_{E/\mathbb{Q}, f}\) such that \(w_{c\sigma} = \tilde{c}w_{\sigma}\) for all \(\sigma \in \Gamma_{E/\mathbb{Q}}\). The universal Taniyama element \(f^E : \Gamma_{E^{ab} / \mathbb{Q}} \rightarrow \mathcal{S}^E(\mathbb{A}_{E,f})/\mathcal{S}^E({E}) \) is defined by 
\[f^E(g) =  \prod_{\sigma \in \Gamma_{E/\mathbb{Q}}} (\sigma^{-1}\mu^E)(w_{g\sigma}^{-1}\tilde{g}w_{\sigma}) \in \mathcal{S}^E(\mathbb{A}_{E,f})/\mathcal{S}^E({E})  \]
where \(\mu^E \in \Hom_{E}({\mathbb{G}_{m}}_{/E}, \mathcal{S}^E_{/E})\) denotes the canonical cocharacter of \(\mathcal{S}^E\) and \(\tilde{g} \in W_{E/\mathbb{Q}, f}\) denotes any lift of \(g \in \Gamma_{E^{ab} / \mathbb{Q}}\). Said differently, if we define \(h_{\sigma}(\tilde{g}) = w_{g\sigma}^{-1}\tilde{g}w_{\sigma} \in \mathbb{A}_{E,f}^{\times}/E^{\times}\) and decompose \[{T}_{E} (\mathbb{A}_{E,f})  / {T}_{E} ({E})  = (E \otimes_{\mathbb{Q}} \mathbb{A}_{E,f})^{\times}/ \, (E \otimes_{\mathbb{Q}} E)^{\times}  \xrightarrow{\sim} \prod_{\sigma}  \mathbb{A}_{E,f}^{\times} / E^{\times} \] with \(\sigma \in \Gamma_{E/\mathbb{Q}} \) then \(f^E(g)\) is nothing more than the image of the element
 \[ h^E(\tilde{g}) \defeq (h_{\sigma^{-1}}(\tilde{g}))_{\sigma} \in \prod_{\sigma } \mathbb{A}_{E,f}^{\times} / E^{\times} \] 
 under the quotient mapping \(T_E \twoheadrightarrow \mathcal{S}^E\).  The element \(f^E(g)\) is well-defined independent of the choice of lift \(\tilde{g}\) or the choices of \(\tilde{c}\) and \(w_{\sigma}\). With our choices of normalisations (we normalise the Artin reciprocity map by insisting that local uniformisers get mapped to geometric Frobenius elements), the level \(E\) Taniyama extension is defined by the map \(\overline{b}^E(g) = f^E(g^{-1})^{-1} \).

 
The precise link between the universal Taniyama element and Tate's half-transfer maps is given by the Hodge theory of CM abelian varieties. Take \(K \subset \overline{\mathbb{Q}}\) a CM field and \(\Phi\) a CM type for \(K\). Let \(T_K = \Res_{K/\mathbb{Q}} \mathbb{G}_m\) and \(\mu_{\Phi}\) denote the cocharacter 
\[ \mu_{\Phi} = \sum_{\phi \in \Phi} [\phi] \in X_{\ast}(T_K)\]
which we can think of as defining the Hodge filtration on a CM abelian variety of type \((K, \Phi)\). Let \( E \subset \overline{\mathbb{Q}}\) be a CM number field which is Galois over \(\mathbb{Q}\) and contains the reflex field of \((K, \Phi)\). Then since \(\mu_{\Phi}\) is defined over \(E\) and has weight \(\mu_{\Phi}^{-(1+c)}\) defined over \(\mathbb{Q}\), there is a unique representation 
\[ \rho_{\Phi}:  \mathcal{S}^E \rightarrow T_K = \Res_{K/\mathbb{Q}} \mathbb{G}_m\] taking the canonical cocharacter \(\mu^E\) of \(\mathcal{S}^E\) to \(\mu_{\Phi}\). This representation is nothing more than the reflex norm of classical CM theory. The map \(f^E\) really is then a `universal Taniyama element' in the sense that whenever \((K, \Phi)\) is a CM type with reflex field contained in \(E\) we have \[\rho_{\Phi}( f^E(g)) =   f_{\Phi}(g) \in  \mathbb{A}_{K, f}/ K^{\times} \subset  T_{K}(\mathbb{A}_{E,f}) / T_K(E) \]
for all \(g\) in \(\Gamma_{\mathbb{Q}}\). In the case when \(E\) actually contains \(K\) we can identify \(\Sigma_K\) with \(\Hom_{\mathbb{Q}}(K,E)\) and decompose 
 \[{T}_{K} (\mathbb{A}_{E,f})  / {T}_{K} ({E}) = (K \otimes_{\mathbb{Q}} \mathbb{A}_{E,f})^{\times}/ (K \otimes_{\mathbb{Q}} E)^{\times}   \xrightarrow{\sim} \prod_{\rho \in \Sigma_K}  \mathbb{A}_{E,f}^{\times} / E^{\times} \] 
The above assertion is equivalent to saying that for any \(\rho : K \hookrightarrow E\) the \(\rho\) component of \(\rho_{\Phi}( f^E(g)) \in{T}_{K} (\mathbb{A}_{E,f})  / {T}_{K} ({E}) \) is equal to the image of \(f_{\Phi}(g) \in \mathbb{A}_{K,f}^{\times}/K^{\times}\) under the map \( \mathbb{A}_{K,f}^{\times}/K^{\times} \rightarrow \mathbb{A}_{E,f}^{\times}/E^{\times}\) induced by \(\rho\). \\

 \textbf{This article.}  From now on we fix a totally real number field \(F \subset \overline{\mathbb{Q}} \) and let \(\Sigma_F = \Hom_{\Qalg}(F, \overline{\mathbb{Q}})\).  In what follows we recover the classical theory when we take \(F\) to be \(\mathbb{Q}\).  Our construction of the \(F\)-plectic Taniyama group is inspired by Nekov\'a\v{r}'s paper \cite{Nekovar}\footnote{It should be pointed out explicitly that Nekov\'a\v{r} reinterprets his results terms of a generalised Taniyama group which is somewhat different to the plectic Taniyama group defined by Theorem 1.2. Our approach, which is motivated by plectic Hodge theory and more closely mirrors Langlands' original construction, is inspired by Nekov\'a\v{r}'s  work but proceeds in a different direction.}, where it is shown that if \(K \subset \overline{\mathbb{Q}}\) is any CM field with totally real subfield \(F\) then Tate's half-transfer map attached to a CM type \(\Phi\) of \(K\) canonically extends to a map 
 \[ \widetilde{F_{\Phi}} : \Aut_F(F\otimes_{\mathbb{Q}} \overline{\mathbb{Q}}) \rightarrow \Gamma_{K}^{ab} \] 
This can now be viewed as evidence for the full plectic conjecture of Nekov\'a\v{r} and Scholl. The first main contribution of this article is to reinterpret Nekov\'a\v{r}'s half-transfer maps along the lines of Tate's original definition. 
 
 \begin{theorem} Let \(w_{\rho}\) be a set of coset representatives for \( \Gamma_K \subset \Gamma_{\mathbb{Q}}\) such that \(w_{c \rho} = cw_{\rho}\) for all \(\rho \in \Sigma_K\). Then Nekov\'a\v{r}'s half-transfer map \(\widetilde{F_{\Phi}}\) sends any \(g \in \Aut_F(F\otimes_{\mathbb{Q}} \overline{\mathbb{Q}})\) to the element 
\[\widetilde{F_{\Phi}}(g) = \prod_{\phi \in \Phi} w_{\alpha_g \rho}^{-1} \alpha_g(w_{\rho})   \;\; \text{ mod } \Gamma_{K^{ab}} \]
of \(\Gamma_K^{ab}\) 
 where \(\alpha_g \in \Gamma_{\mathbb{Q}} \# \Gamma_{F} \subset \Gamma_{\mathbb{Q}} \# \Gamma_{K} \) is to the image of \(g\) under the canonical isomorphism \(\Aut_F(F\otimes_{\mathbb{Q}} \overline{\mathbb{Q}}) \xrightarrow{\sim} \Gamma_{\mathbb{Q}} \# \Gamma_{F}\). 
  \end{theorem}

In the remainder of this paper we take Theorem 1.1 as our starting point and revisit Langlands' original definition to construct an \(F\)-plectic Taniyama group for any totally real number field \(F\). This takes the form of an extension                     \[ 1 \rightarrow \mathcal{S}_F \rightarrow \mathcal{T}_F \rightarrow \Aut(F \otimes_{\mathbb{Q}} \overline{\mathbb{Q}}) \rightarrow 1 \]
of the \(F\)-plectic absolute Galois group \(\Aut(F \otimes_{\mathbb{Q}} \overline{\mathbb{Q}})\) by the \(F\)-plectic Serre group \(\mathcal{S}_F\), together with a continuous finite-adelic splitting \(s_F :  \Aut(F \otimes_{\mathbb{Q}} \overline{\mathbb{Q}}) \rightarrow  \mathcal{T}_F(\mathbb{A}_{\mathbb{Q}, f}) \). Here \(\mathcal{S}_F\) denotes the \(F\)-plectic Serre group \(\mathcal{S}^{\Sigma_F}\), which plays the role of \(\mathcal{S}\) in the theory of \(F\)-plectic Hodge structures. The group \(\mathcal{T}_F\) is built from extensions of the form
                      \[ 1 \rightarrow \mathcal{S}_F^E \rightarrow \mathcal{T}_F^E \rightarrow \Aut_F(F\otimes_{\mathbb{Q}} E^{ab}) \rightarrow 1\] together with continuous finite-adelic splittings \(s_F^E :  \Aut_F(F\otimes_{\mathbb{Q}} E^{ab}) \rightarrow  \mathcal{T}_F^E(\mathbb{A}_{\mathbb{Q}, f}) \). Here \(E \subset \overline{\mathbb{Q}}\) is a CM field containing \(F\) and Galois over \(\mathbb{Q}\), while \(\mathcal{S}_F^E\) denotes the torus \((\mathcal{S}^E)^{\Sigma_F}\). Any such extension gives rise to and is completely determined by a map 
                      \[\overline{b}_F^E :     \Aut_F(F\otimes_{\mathbb{Q}} E^{ab})\rightarrow \mathcal{S}_F^E(\mathbb{A}_{E,f})/\mathcal{S}_F^E(E) \]
and it is possible (see Proposition 6.1 for a precise statement) to write down certain conditions which, when satisfied, guarantee that such a map does indeed come from an extension of the above form. 

For \(E\) as above let \(W_{E/\mathbb{Q}, f}\) denote the finite-adelic version of the Weil group of the extension \(E / \mathbb{Q}\). As with Langlands' construction, we pick a lift \(\tilde{c} \in W_{E/\mathbb{Q},f}\) of complex conjugation and a set \(\{w_{\sigma}\}\) of coset representatives for \(\mathbb{A}_{E,f}^{\times}/E^{\times} \subset W_{E/\mathbb{Q}, f}\) such that \(w_{c\sigma} = \tilde{c}w_{\sigma}\) for all \(\sigma \in \Gamma_{E/\mathbb{Q}}\)\footnote{This is very slight simplification of the choices involved. See \S5 for a more detailed exposition of the construction.
}.
  An embedding \(j : F \hookrightarrow E\) induces a natural map \( j : \Aut_F(F\otimes E^{ab}) \hookrightarrow \Aut_E(E \otimes E^{ab})\) and hence a map \(j:  \Gamma_{E^{ab}/\mathbb{Q}} \# \Gamma_{E^{ab}/F}  \rightarrow \Gamma_{E^{ab}/\mathbb{Q}} \# \Gamma_{E^{ab}/ E }\). This can be expressed in purely group theoretic terms, and as a result we have a similar map on the level of Weil groups. We view the torus \(\mathcal{S}_F^E = (\mathcal{S}^E)^{\Sigma_F}\) as a quotient of \(T_{F \otimes_{\mathbb{Q}} E} \xrightarrow{\sim} T_E^{\Sigma_F}\) and can decompose \[{T}_{F \otimes_{\mathbb{Q}} E} (\mathbb{A}_{E,f})  / {T}_{F \otimes_{\mathbb{Q}} E}({E})  = (F \otimes_{\mathbb{Q}} E \otimes_{\mathbb{Q}} \mathbb{A}_{E,f})^{\times}/ \, (F \otimes_{\mathbb{Q}} E \otimes_{\mathbb{Q}} E)^{\times}  \xrightarrow{\sim} \prod_{ j \otimes \sigma}  \mathbb{A}_{E,f}^{\times} / E^{\times} \] where \(j \otimes \sigma\) denotes the component on which \(F\) acts by \(j \in \Sigma_F\) and \(E\) acts by \(\sigma \in \Gamma_{E/\mathbb{Q}} \). The key construction in this paper is that of a particular map \[f_F^E : \Aut_F(F\otimes E^{ab}) \rightarrow \mathcal{S}_F^E(\mathbb{A}_{E,f})/\mathcal{S}_F^E({E}) \]  which we call the  \(F\)-plectic universal Taniyama element. This is defined by sending \(g \in \Aut_F(F\otimes E^{ab})\) to the image of the element
\[ h_F^E( \tilde{\alpha}_g) \defeq (h_{\sigma^{-1}}(j(\tilde{\alpha}_g)))_{j \otimes \sigma} \defeq (w_{ j(\tilde{\alpha}_g)\sigma^{-1}}^{-1} j(\tilde{\alpha}_g)(w_{\sigma^{-1}}))_{j \otimes \sigma}    \in   {T}_{F\otimes E} (\mathbb{A}_{E,f})  / {T}_{F\otimes E} ({E})   \]
under \(T_{F \otimes E} \twoheadrightarrow \mathcal{S}_F^E\), where \(\tilde{\alpha}_g\) is any lift of \(\alpha_g\) through the map \(W_{E/\mathbb{Q} , f} \# W_{E/F , f} \twoheadrightarrow \Gamma_{E^{ab}/\mathbb{Q}} \# \Gamma_{E^{ab}/ F}\). To make it clear, here each \(j(\tilde{\alpha}_g)\) is a element of the group \[W_{E/\mathbb{Q} , f} \# \mathbb{A}_{E,f}^{\times}/E^{\times} = \{ \alpha : W_{E/\mathbb{Q} , f} \xrightarrow{\sim}W_{E/\mathbb{Q} , f} \text{ s.t. for all } w \in W_{E/\mathbb{Q} , f}, e \in \mathbb{A}_{E,f}^{\times}/E^{\times} \text{ we have } \alpha(w{e}) = \alpha(w) {e} \}\] 
Any such element induces a permutation \(\sigma \mapsto j(\tilde{\alpha}_g) \sigma\) of the quotient \(\Gamma_{E/\mathbb{Q}}\) such that \(j(\tilde{\alpha}_g)(w_{\sigma^{-1}})\) lies in the same coset as \(w_{ j(\tilde{\alpha}_g) \sigma^{-1}}\) for all \(\sigma \in \Gamma_{E/\mathbb{Q}}\). The difference \(w_{ j(\tilde{\alpha}_g)\sigma^{-1}}^{-1} j(\tilde{\alpha}_g)(w_{\sigma^{-1}})\) is then an element of \(\mathbb{A}_{E,f}^{\times}/E^{\times}\). The element \[ h_F^E( \tilde{\alpha}_g) \in {T}_{F\otimes E} (\mathbb{A}_{E,f})  / {T}_{F\otimes E} ({E})\]
simply stores this collection of elements of  \(\mathbb{A}_{E,f}^{\times}/E^{\times}\) for each \(j \in \Sigma_F\) and \(\sigma \in \Gamma_{E/\mathbb{Q}} \) in an appropriate way. The content is that this becomes well-defined independent of our choices when we pass to the quotient \({T}_{F\otimes E}  \twoheadrightarrow \mathcal{S}_F^E\).  What's more, this
\(F\)-plectic universal Taniyama element  goes on to satisfy the conditions of Proposition 6.1 and so defines an extension that we call the \(F\)-plectic Taniyama group. 

\begin{theorem} For any totally real field \(F\) and CM number field \(E \subset \overline{\mathbb{Q}}\) which splits \(F\):
\begin{enumerate}

\item For each \(g \in \Aut_F(F\otimes E^{ab})\) the element \(f_F^E(g)\) is well-defined independent of the choice of lift \(\tilde{\alpha}_g\) and the choices of \(\tilde{c}\) and of coset representatives \(w_{\sigma}\) for \(\sigma \in \Gamma_{E/\mathbb{Q}}\). 

\item There is a unique extension \[ 1 \rightarrow \mathcal{S}_F^E \rightarrow \mathcal{T}_F^E \rightarrow  \Aut_F(F\otimes E^{ab})  \rightarrow 1\]
together with a continuous finite-adelic splitting \(s^E_F :  \Aut_F(F\otimes E^{ab}) \rightarrow \mathcal{T}_F^E(\mathbb{A}_{\mathbb{Q}, f}) \) giving rise to the map \(\overline{b}_F^E : \Aut_F(F\otimes_{\mathbb{Q}} E^{ab})  \rightarrow \mathcal{S}_F^E(\mathbb{A}_{E,f})/S_F^E(E) \) defined by \(g \mapsto f_F^E(g^{-1})^{-1}\). 

\item This family of extensions pieces together produce an extension of the form \[ 1 \rightarrow \mathcal{S}_F \rightarrow \mathcal{T}_F \rightarrow \Aut_F(F\otimes_{\mathbb{Q}} \overline{\mathbb{Q}}) \rightarrow 1  \]
  together with a continuous finite-adelic splitting \(s_F :  \Aut_F(F\otimes_{\mathbb{Q}}\overline{\mathbb{Q}})  \rightarrow  \mathcal{T}_F(\mathbb{A}_{\mathbb{Q}, f}) \).  \end{enumerate}

\end{theorem}

We end by making clear the relationship between our \(F\)-plectic Taniyama group and Nekov\'a\v{r}'s extended half-transfer maps.    Let \(K\) be a CM field with maximal totally real subfield \(F\). Then given a CM type \(\Phi\) on \(K\), we can write \(\Phi = \{ \Phi_j : j \in \Sigma_F\}\) where \(\phi_j\) is the unique element of \(\Phi\) such that \(\phi_j|_F = j\). We can think of the cocharacter \(\mu_{\Phi}  = \sum_{\phi \in \Phi} [\phi]\) as the restriction to the diagonal of the homomorphism \[  \mu_{\Phi, F} = (\mu_{j})_{j\in \Sigma_F} : \mathbb{G}_{m, \mathbb{C}}^{\Sigma_F} \rightarrow T_{K, \mathbb{C}} \] 
where for each \(j \in \Sigma_F\) we have \(\mu_{j} = [\phi_j] \in X_{\ast}(T_K)\). This can be thought of as using the real multiplication by \(F\) to define an \(F\)-plectic Hodge structure on a CM abelian variety of type \((K, \Phi)\). If \(E \subset \overline{\mathbb{Q}}\) is a CM number field containing \(K\) and Galois over \(\mathbb{Q}\) then there is a unique representation \(\rho_j : T_E \rightarrow T_K\) taking the canonical cocharacter \(\mu^E\) of \(T_E\) to \(\mu_j\). These piece together to give a representation \[\rho_{\Phi, F} = (\rho_j)_j: T_{F\otimes_\mathbb{Q}E} \xrightarrow{\sim} T_E^{\Sigma_F} \rightarrow T_K\] whose restriction to the diagonal is the classical reflex norm. It important to note that \(\rho_{\Phi, F}\) itself need not factor through the quotient \(T_{F\otimes_\mathbb{Q}E} \twoheadrightarrow \mathcal{S}_F^E\). Now the universal Taniyama element \(f_F^E\) was constructed by looking at the image of 
\[ h_F^E(\tilde{\alpha}_g) \defeq (h_{\sigma^{-1}}(j(\tilde{\alpha}_g)))_{j \otimes \sigma} = (w_{ j(\tilde{\alpha}_g)\sigma^{-1}}^{-1} j(\tilde{\alpha}_g)(w_{\sigma^{-1}}))_{j \otimes \sigma}    \in   {T}_{F\otimes E} (\mathbb{A}_{E,f})  / {T}_{F\otimes E} ({E})   \]
under the map \(T_{F \otimes E} \twoheadrightarrow \mathcal{S}_F^E\). If instead of mapping down to this quotient, we instead look at the image of this element under \(\rho_{\Phi, F}\), we obtain an element
\[ \rho_{\Phi, F}(h_F^E(\tilde{\alpha}_g)) \in  {T}_{K} (\mathbb{A}_{E,f})  / {T}_{K} ({E}) = (K \otimes_{\mathbb{Q}} \mathbb{A}_{E,f}^{\times} )/ (K \otimes_{\mathbb{Q}} E)^{\times}  \xrightarrow{\sim} \prod_{\rho : K \hookrightarrow E} \mathbb{A}_{E,f}^{\times}/E^{\times}   \] 
In general this element may depend both on the choice of lift \(\tilde{\alpha}_g\) of \(\alpha_g\) and possibly also on the choices we made in the construction. Regardless, it follows from Theorem 1.1 that for any \(\rho: K \hookrightarrow E\) the image of the \(\rho\) component of \(\rho_{\Phi, F}(h_F^E(\tilde{\alpha}_g)) \in  {T}_{K} (\mathbb{A}_{E,f})  / {T}_{K} ({E})\) under the Artin reciprocity map \(\mathbb{A}_{E,f}^{\times} / E^{\times} \twoheadrightarrow \Gamma_E^{ab}\) is always equal to the image of \(\widetilde{F_{\Phi}}(g) \in \Gamma_K^{ab}\) under the map \(\Gamma_K^{ab} \rightarrow \Gamma_E^{ab}\) induced by \(\rho\). This is a direct analogue of the relationship between Langlands' universal Taniyama element and Tate's half-transfer maps.

\section{Plectic Galois theory}

Let \(F\) be any number field and write \(\Sigma_F\) for the set \(\Hom_{\Qalg}(F, \overline{\mathbb{Q}})\). We will say a subfield \(k\) of \(\overline{\mathbb{Q}}\) splits \(F\) if it contains the image of \(F\) under any embedding \(j \in \Sigma_F\). In this case the map \( a\otimes b \mapsto (j(a)b)_{j \in \Sigma_F}\) defines an isomorphism of \(F\)-algebras
\[F \otimes_{\mathbb{Q}} k  \xrightarrow{\sim} \prod_{j  \in {\Sigma_F}} k   \]
where the action of \(F\) on the \(j\) component of the right hand side is via the embedding \(j : F \hookrightarrow k\). When \(k\) is Galois over \(\mathbb{Q}\) then we call the group  \(\Aut_F(F\otimes_{\mathbb{Q}} k )\) of all automorphisms of the \(F\)-algebra \(F\otimes_{\mathbb{Q}} k\) the \(F\)-plectic Galois group of \(k\).  From the decomposition above it is clear that \(\Aut_F(F\otimes_{\mathbb{Q}} k )\) acts on the set \(\Sigma_F\). 

When \(F\) is a subfield of \(k\) we can interpret the \(F\)-plectic Galois group \(\Aut_F(F\otimes_{\mathbb{Q}} k )\) as a wreath product. In this case we have  \(\Gamma_{k / F} := \Gal(k/F) \subset \Gamma_{k / \mathbb{Q}} := \Gal(k/\mathbb{Q}) \) and we can identify the set \(X = \Gamma_{k / \mathbb{Q}}/ \Gamma_{k / F}\) with \(\Sigma_F\).  If we consider the ring \(k^X\) as an \(F\)-algebra where the \(F\)-algebra structure comes from \(F \subset k\) on each factor then there is a natural isomorphism (see \cite{Nekovar}, Prop. 1.1.3)
\(  S_X \ltimes \Gamma_{k / F}^X \xrightarrow{\sim} \Aut_F(k^X) \)
where the semidirect product  is taken with respect to the natural permutation action of the symmetric group \(S_X\) on \(\Gamma_{k / F}^X\). This is given by mapping \((\pi, (h_x)_{x \in X}) \in S_X \ltimes \Gamma_{k / F} ^X\) to the automorphism of \(k^X\) which sends
\( (a_x)_{x \in X} \mapsto (b_x)_{x \in X}\)
where \( b_{\pi(x)} = h_x(a_x) \). Finally, if we fix a set \(s = (s_x)_{x \in X}\) of coset representatives for \( \Gamma_{k / F}\) in \(\Gamma_{k / \mathbb{Q}}\) then the map \(s :  k^{X} \rightarrow \prod_{x  \in X} k  \)
 given by \( (a_x)_{x \in X} \mapsto (s_x(a_x))_{x \in X}\) is an isomorphism of \(F\)-algebras and induces an isomorphism
\[ \beta_s :  \Aut_F(F\otimes_{\mathbb{Q}} k) \xrightarrow{\sim} S_X \ltimes \Gamma_{k / F}^X \]
of groups. If each coset representative \(s_x\) is replaced by \(s^{\prime}_x = s_xt_x\) with \({t} = (t_x)_{x \in X} \in \Gamma_{k / F}^X\) then
\(  \beta_{s^{\prime}}(g) = (1, {t})^{-1} \beta_s(g) (1, {t})  \in S_X \ltimes \Gamma_{k / F}^X \)
for any \(g \in  \Aut_F(F\otimes_{\mathbb{Q}} k)\).

There is a canonical injective group homomorphism \(\Gamma_{k/ \mathbb{Q}} \hookrightarrow \Aut_F(F\otimes_{\mathbb{Q}} k )\) given by sending \(g \in \Gamma_{k/ \mathbb{Q}} \) to \(id_F \otimes g \in  \Aut_F(F\otimes_{\mathbb{Q}} k)\). More generally the construction of plectic Galois groups is functorial in \(F\) in the sense that if \(F\) and \(F^{\prime}\) are both split by \(k\) then a map \(F \hookrightarrow F^{\prime}\) of number fields induces a map 
\( \Aut_{F}(F \otimes_{\mathbb{Q}} k) \hookrightarrow \Aut_{F^{\prime}}(F^{^{\prime}}\otimes_{\mathbb{Q}} k) \) of plectic Galois groups. When \(F \subset F^{\prime} \subset k\) this can be described explicitly in terms of wreath products as follows:

\begin{proposition} Let \(k\) be Galois over \(\mathbb{Q}\) and \(F \subset F^{\prime} \subset k\).  As above let us write \(X\) for the set \( \Gamma_{k / \mathbb{Q}} / \Gamma_{k/F} = \Sigma_F\) but also let \(X^{\prime} \) denote \(\Gamma_{k / \mathbb{Q}} / \Gamma_{k/F^{\prime}} = \Sigma_{F^{\prime}} \) and \(Y\) denote \( \Gamma_{k / F} / \Gamma_{k/F^{\prime}} = \Hom_{F}(F^{\prime}, k) \).  If we pick coset representatives \((s_x)_{x \in X}\) for \( \Gamma_{k/F} \) in \(\Gamma_{k / \mathbb{Q}}\) and \((t_y)_{y \in Y}\) for \( \Gamma_{k/F^{\prime}} \) in \(\Gamma_{k/F} \) then \( (s^{\prime}_{x^{\prime}} ) = \{s_xt_y : x \in X, y\in Y\}\) is a set of coset representatives for \(\Gamma_{k/F^{\prime}} \) in \(\Gamma_{k/\mathbb{Q}} \). 
If we write \(\beta_s(g) = (\pi, (h_x)_{x\in X}) \in S_X \ltimes  \Gamma_{k/F}^X\) then for any \(g \in \Aut_{F}(F \otimes_{\mathbb{Q}} k)\) we have \[\beta_{s^{\prime}}( id_{F^{\prime}} \otimes g) = (\pi^{\prime}, (h^{\prime}_{x^{\prime}})_{x^{\prime} \in X^{\prime}}) \in S_{X^{\prime}} \ltimes \Gamma_{k/F^{\prime}}^{X^{\prime}}\]  with
\(\pi^{\prime}(s_xt_y|_{F^{\prime}}) = s_{\pi(x)}t_{h_x y} |_{F^{\prime}}\) and \(  h^{\prime}_{s_xt_y|_{F^{\prime}}} = t_{h_xy}^{-1}h_xt_y \).
\begin{proof} This follows from \cite{Nekovar}, Proposition 1.1.8 and (the proof of) Proposition 2.2.4. 
\end{proof}
\end{proposition}

Now let \(u \in \Gamma_{k/\mathbb{Q}}\) and \(F^{\prime} = u(F) \subset k\). Then \(u|_F : F \xrightarrow{\sim} F^{\prime}\) induces an isomorphism 
\( [u] : \Aut_{F}(F \otimes_{\mathbb{Q}} k) \rightarrow \Aut_{F^{\prime}}(F^{\prime} \otimes_{\mathbb{Q}} k) \)
which in this case is just given by 
\( g \mapsto (u \otimes id_k) \circ g \circ (u^{-1} \otimes id_{k} ) \). Again we have the following explicit description:

\begin{proposition}  Let \(F \subset k\), \(u \in \Gamma_{k/\mathbb{Q}}\) and \(F^{\prime} = u(F) \subset k\).  We have a bijection \(X \xrightarrow{\sim} X^{\prime}\) given by \(x \mapsto x^{\prime} = xu^{-1}\) (in terms of embeddings this is just  the natural identification of \(\Sigma_F\) with \(\Sigma_{F^{\prime}}\) induced by \(u|_F\)). If we choose coset representatives \(s = (s_x)_{x \in X}\) for  \(\Gamma_{k/F}\) in \(\Gamma_{k/\mathbb{Q}}\) then we can define coset representatives for  \(\Gamma_{k/F^{\prime}}\) in \(\Gamma_{k/\mathbb{Q}}\)  by \(s^{\prime}_{x^{\prime}} = s_x u^{-1}\).  With these choices, there is a commutative diagram 
\[ \begin{CD}  \Aut_{F}(F \otimes k) @> \beta_s>> S_{X} \ltimes \Gamma_{k/F}^{X}  \\
@V [u] VV @V u_{\ast} VV \\
 \Aut_{F^{\prime}}(F^{\prime} \otimes k) @>\beta_{s^{\prime}} >> S_{X^{\prime}} \ltimes \Gamma_{k/F^{\prime}}^{X^{\prime}} \end{CD}  \]
where here  \(u_{\ast} : S_X \ltimes \Gamma_{k/F}^X \rightarrow S_{X^{\prime}} \ltimes \Gamma_{k/F^{\prime}}^{X^{\prime}} \)
denotes the map \( (\pi, (h_x)_x) \mapsto (\pi^{\prime}, (h^{\prime}_{x^{\prime}})_{x^{\prime}}) \) where 
\( \pi^{\prime}(x^{\prime}) = \pi(x)^{\prime}\) (i.e. \(  \pi^{\prime}(xu^{-1}) = \pi(x)u^{-1}) \)
and 
\( h^{\prime}_{x^{\prime}} = u h_x u^{-1} \).  

\begin{proof} This is \cite{Nekovar}, Proposition 1.1.6 and Proposition 1.1.7(ii).
\end{proof}
\end{proposition}

Plectic Galois groups can be understand in purely group-theoretic terms. To see this, let \(G\) be any group and \(H \subset G\) a subgroup of finite index. We write \(X\) for the set \(G / H\) of left cosets of \(H \) in \(G\). Then \(H\) acts on the right on the set \(G\) by right translation and an orbit of \(H\) under this action is nothing more than an element of \(X\).  

\begin{definition}[Nekov\'a\v{r}, Scholl - cf. \cite{SN}, \S3] The plectic group of the pair \((G,H)\) is the group
 \( G \# H \defeq \Aut_{\rsets H}G \)
 of automorphisms of \(G\) as a right \(H\)-set. More explicitly, this consists of the following collection of set-theoretic bijections
 \[ G \# H = \{ \alpha : G \xrightarrow{\sim} G  \text{ such that for all } g \in G, h\in H \text{ we have } \alpha(gh) = \alpha(g) h \} \]
 Note there is a natural inclusion \( G \hookrightarrow G \# H\) given by mapping an element \(g \in G\) to the left translation by \(g\) map \(L_g : G \rightarrow G \).
\end{definition} 

 Any \(\alpha \in G \# H\) gives rise to a well defined permutation of the set \(X = G/H\). We will denote this permutation by \(\pi = \pi(\alpha) \in S_X\) but where the context is clear we will often ease the notation by writing this simply as \(x \mapsto \alpha(x)\),  or alternatively as \(x \mapsto \alpha x\), for all \(x \in X\). 
This plectic group can also be interpreted in terms of a wreath product. To see this,  let us choose a collection \(s = (s_x)_{x\in X}\) of coset representatives for \(H\) in \(G\). The permutation \(\pi\) is defined by the fact that \( \alpha(s_x)H = \pi(s_xH) = s_{\alpha x}H\) and so the elements
\[  h_x = h_{x}(\alpha) :=  s_{\alpha x}^{-1}\alpha(s_x) \]
lie in \( H \subset G\). After making these choices we can define an isomorphism  \[ \rho_s :  G \# H \rightarrow S_X \ltimes H^X \] 
by sending  \(\alpha  \in G \# H\) to \((\pi, (h_x(\alpha))_{x \in X}) \in S_X \ltimes H^X\). As before, the semidirect product on the right hand side is taken with respect to the permutation action of \(S_X\) on \(H^X\). If the choice of coset representatives \(s = (s_x)_{x \in X}\) is changed to \(s^{\prime} = ( s_xt_x)_{x\in X} \)  for some \({t} = (t_x)_{x \in X} \in H^X\) then the maps \(\rho_s\) and \(\rho_{s^{\prime}}\) are related by the formula 
\( \rho_{s^{\prime}}(\alpha) = (1, {t})^{-1} \rho_s(\alpha) (1, {t}) \)
for all 
\(\alpha \in  G \# H\).

If \(H^{\prime} \subset H \subset G\) are subgroups of finite index then \(G \# H \) is a contained in \(G \# H^{\prime}\). The following proposition describes this inclusion map explicitly in terms of wreath products:

\begin{proposition} Suppose  \(G \supset H \supset H^{\prime}\) are subgroups of finite index and let \(X\) denote the set \(G/H\) and similarly let \(X^{\prime}\) denote \( G/H^{\prime}\) and \(Y \) denote  \(H/H^{\prime}\). Pick coset representatives \((s_x)_{x \in X}\) for \(H\) in \(G\) and \((t_y)_{y \in Y}\) for \(H^{\prime}\) in \(H\). Then \(s^{\prime} = \{s_xt_y : x \in X, y\in Y\}\) is a set of coset representatives for \(H^{\prime}\) in \(G\). If we write \(\rho_s(\alpha) = (\pi, (h_x)_{x\in X}) \in S_X \ltimes H^X\) then for any \(\alpha \in  G \# H \subset G \# H^{\prime}\) we have \[\rho_{s^{\prime}}(\alpha) = (\pi^{\prime}, (h^{\prime}_{x^{\prime}})_{x^{\prime} \in X^{\prime}}) \in S_{X^{\prime}} \ltimes (H^{\prime})^{X^{\prime}}\] where 
\( \pi^{\prime}(s_xt_yH^{\prime}) = s_{\pi(x)}t_{h_x y} H^{\prime} \)  and \(  h^{\prime}_{s_xt_yH^{\prime}} = t_{h_xy}^{-1}h_xt_y \) 
\begin{proof} We first compute 
\[ s_{\pi(x)} t_{h_x y} H^{\prime} =  s_{\pi(x)} h_x t_y H^{\prime} = s_{\pi(x) } s_{\pi(x)}^{-1} \alpha(s_x) t_y H^{\prime} =  \alpha(s_x t_y)H^{\prime} = \pi^{\prime}(s_xt_yH^{\prime}) \]
as \(t_y \in H\) and  \(\alpha \in G \# H \subset G \# H^{\prime}\). Thus 
\[h^{\prime}_{s_xt_yH^{\prime}} =  (s_{\pi(x)}t_{h_x y})^{-1} \alpha(s_x t_y) = t_{h_x y}^{-1}s_{\pi(x)}^{-1} \alpha(s_x) t_y = t_{h_xy}^{-1}h_xt_y  \]
as claimed. 
\end{proof}
\end{proposition}

Now fix \(u \in G\) and let \(H^{\prime}\) denote the conjugate subgroup \(uHu^{-1}\).  Then \(u\) gives rises to a map 
\( [u] : G \# H \rightarrow G \# H^{\prime} \)
given by \(\alpha \mapsto \alpha^{\prime}\) where for all \(g \in G\) we have \(\alpha^{\prime}(g) = \alpha(gu)u^{-1} \). The map \([u]\) only depends on the left coset \(uH\) and can be described in terms of wreath products as follows:

\begin{proposition} Fix \(u \in G\) and let \(H^{\prime}\) denote the conjugate subgroup \(uHu^{-1}\). We have a bijection \(X \xrightarrow{\sim} X^{\prime} = G / H^{\prime}\) given by \(x \mapsto x^{\prime} = xu^{-1} =  \{ gu^{-1} : g \in X\} \). Moreover if we choose coset representatives \(s = (s_x)_{x \in X}\) for \(H\) in \(G\) then we can define coset representatives for \(H^{\prime} \) in \(G\) by \(s^{\prime}_{x^{\prime}} = s_x u^{-1}\). With these choices, there is a commutative diagram
 \[ \begin{CD} G \# H @> \rho_s>> S_X \ltimes H^X \\
@V [u] VV @V u_{\ast} VV \\
G \# H^{\prime} @>\rho_{s^{\prime}} >> S_{X^{\prime}} \ltimes (H^{\prime})^{X^{\prime}} \end{CD}  \]
where here  \(u_{\ast} : S_X \ltimes H^X \rightarrow S_{X^{\prime}} \ltimes (H^{\prime})^{X^{\prime}} \)
denotes the map \( (\pi, (h_x)_x) \mapsto (\pi^{\prime}, (h^{\prime}_{x^{\prime}})_{x^{\prime}}) \) where 
\( \pi^{\prime}(x^{\prime}) = \pi(x)^{\prime}\) (i.e. \(  \pi^{\prime}(xu^{-1}) = \pi(x)u^{-1}) \)
and 
\( h^{\prime}_{x^{\prime}} = u h_x u^{-1} \).  
\begin{proof} 
 
%

The proof is a straightforward calculation using the the definitions.
\end{proof}
\end{proposition} 

From these computations, we conclude: 

\begin{theorem}  Let \(k \subset \overline{\mathbb{Q}}\) be Galois over \(\mathbb{Q}\) and \( F \subset k\) a number field. A choice of coset representatives \(s\) for \(\Gamma_{k / F} \subset \Gamma_{k / \mathbb{Q}} \) defines a composite isomorphism \[ \Aut_F(F\otimes_{\mathbb{Q}} k) \xrightarrow{\beta_s} S_X \ltimes H^X \xrightarrow{\rho_s^{-1}}  \Gamma_{k / \mathbb{Q}} \# \Gamma_{k / F}\]
which is independent of the choice of \(s\). We will denote this canonical isomorphism \( \Aut_F(F\otimes_{\mathbb{Q}} k) \xrightarrow{\sim}   \Gamma_{k / \mathbb{Q}} \# \Gamma_{k / F}\) by \(g \mapsto \alpha_g\). Moreover, 
\begin{itemize} 
\item if \(F \subset F^{\prime} \subset k\) is another number field then the natural map \(\Aut_F(F\otimes_{\mathbb{Q}} k) \hookrightarrow \Aut_{F^{\prime}}(F^{\prime} \otimes_{\mathbb{Q}} k)\) corresponds to the inclusion \(  \Gamma_{k/\mathbb{Q}} \# \Gamma_{k/F}  \hookrightarrow  \Gamma_{k/\mathbb{Q}} \# \Gamma_{k/F^{\prime}} \)
\item Similarly, for \(u \in \Gamma_{k/ \mathbb{Q}}\), the natural map \([u] : \Aut_F(F\otimes_{\mathbb{Q}} k) \xrightarrow{\sim} \Aut_{F^{\prime}}(F^{\prime} \otimes_{\mathbb{Q}} k) \) induced by \(u|_F : F \xrightarrow{\sim} F^{\prime} := u(F)\) corresponds to the map \( [u] : \Gamma_{k/\mathbb{Q}} \# \Gamma_{k/F}  \xrightarrow{\sim}  \Gamma_{k/\mathbb{Q}} \# \Gamma_{k/F^{\prime}} \) from Proposition 2.4. 
\end{itemize}
\begin{proof} The first observation is due to Nekov\'a\v{r} and Scholl (cf. \cite{SN}, \S3). Indeed, we have already noted that if the choice of \(s = (s_x)_{x \in X}\) is changed to \(s^{\prime} = ( s_xt_x)_{x\in X} \) for some \({t} = (t_x)_{x \in X} \in H^X\) then we have both
\( \beta_{s^{\prime}}(g) = (1, {t})^{-1} \beta_s(g) (1, {t}) \)
for all \(g \in \Aut_F(F\otimes_{\mathbb{Q}} k)\) and 
\( \rho_{s^{\prime}}(\alpha) = (1, {t})^{-1} \rho_s(\alpha) (1, {t}) \)
for all
\(\alpha \in  \Gamma_{k / \mathbb{Q}} \# \Gamma_{k / F}\), which proves the first claim. The final two claims follow from combining Proposition 2.1 and Proposition 2.3 and then Proposition 2.2 and Proposition 2.4 respectively. 
\end{proof} 
\end{theorem} 

Suppose now that \(G\) is a topological group and \(H \subset G\) is an open subgroup of finite index. As discussed above, a choice of coset representatives \(s= (s_x)_{x \in X}\) for \(H\) in \(G\) defines an isomorphism
\(\rho_s : G \# H \xrightarrow{\sim} S_X \ltimes H^X\).
The two groups \(H^{X}\) (with the product topology) and \(S_X\) (with the discrete topology) are topological groups and the action of \(S_X\) on \(H^{X}\) is continuous. Thus
\(S_X \ltimes H^X\)
carries the structure of a topological group. We define a topology on \(G \# H\) by declaring \(U \subset G \# H\) to be open if and only if \(\rho_s(U) \subset S_X \ltimes H^X\) is open.
This is well defined since if we choose a different section \(s^{\prime} = s t\) with \(t = (t_x)_{x \in X} \in H^X\) then 
\(\rho_{s^{\prime}}(U) = (1, t)^{-1}\rho_s(U)(1, t)\)
is open if and only if \( \rho_{s}(U)  \) is. It follows from the explicit calculations in Proposition 2.3 and Proposition 2.4 that with this topology: \begin{itemize}
\item the inclusion \(G \hookrightarrow  G \# H\) is continuous;
\item if \(H^{\prime} \subset H \subset G\) are open subgroups of finite index the inclusion \( G \# H \hookrightarrow G \# H^{\prime} \) is continuous; and 
\item if \(u \in G\) and \(H^{\prime} = uHu^{-1}\) then the map 
\( [u] : G \# H \xrightarrow{\sim} G \# H^{\prime} \)
is a homeomorphism. 
\end{itemize}

\section{Plectic CM theory} 

Let \(K \subset \overline{\mathbb{Q}}\) be a CM field and \(\Sigma_{K} = \Hom_{\Qalg} (K, \mathbb{\overline{Q}}) = \Gamma_{\mathbb{Q}} / \Gamma_{K} \). Let \(F = K^+\) denote the maximal totally real subfield of \(K\). In his work on hidden symmetries in the theory of complex multiplication \cite{Nekovar}, Nekov\'a\v{r} has defined a natural map 
 \[ \widetilde{F_{\Phi}} : \Aut_F(F \otimes_{\mathbb{Q}} \overline{\mathbb{Q}}) \rightarrow \Gamma_K^{ab} \]
 extending Tate's half-transfer 
 \(  {F_{\Phi}} : \Gamma_{\mathbb{Q}} \rightarrow \Gamma_{K}^{ab} \) for any CM type \(\Phi\) of \(K\). 
 
 Following Nekov\'a\v{r}'s notation let us write \(X = \Sigma_F = \Gamma_{\mathbb{Q}} / \Gamma_F\) and fix a section \(s\) given by \(x \mapsto s_x\) to the restriction map \( \Gamma_{\mathbb{Q}} \twoheadrightarrow X = \Hom_{\Qalg}(F ,\overline{\mathbb{Q}}) \). 
We note that \(\{s_x|_K : x\in X \}\) gives a distinguished CM type of \(K\), and so we get the following data:
\begin{itemize} 
\item A bijection between \((\mathbb{Z} / 2\mathbb{Z})^X\) and the set of CM types of \(K\) sending a collection \( a = (a_x)_{x \in X} \) to \(\{ \phi_x = c^{a_x}s_x|_K : x \in X\}\). 
\item A section \(\rho \mapsto w_{\rho} \) to \(\Gamma_{\mathbb{Q}} \twoheadrightarrow \Sigma_K = \Hom_{\Qalg}(K ,\overline{\mathbb{Q}}) \) such that \(w_{c\rho} = cw_{\rho}\) by setting
\[ w_{c^b s_x|_K} = c^{b} s_x\]
\end{itemize}

For \(h \in \Gamma_F \) we define \(\overline{h} \in \mathbb{Z}/2\mathbb{Z}\) by the relation \(h|_K = c^{\overline{h}(x)}  \in \Gal(K/F)\). 

\begin{proposition}(Nekov\'a\v{r})  Fix \(s\) as above and let \(\Phi\) be a CM type corresponding to \(a \in (\mathbb{Z} / 2\mathbb{Z})^X\). Define a map
\( _s\widetilde{F_{\Phi}} : S_X \ltimes \Gamma_F^X \rightarrow \Gamma_K^{ab}\)
by the formula
\[ (\pi, (h_x)_{x\in X}) \mapsto \prod_{x\in X} s_{\pi(x)}^{-1} c^{a_x + \overline{h_x}} s_{\pi(x)} h_x s_x^{-1} c^{a_x} s_x |_{K^{ab}} \] 
Then the composite map 
\[ \widetilde{F_{\Phi}} : \Aut_F(F\otimes_{\mathbb{Q}} \overline{\mathbb{Q}} ) \xrightarrow{\beta_s} S_X \ltimes \Gamma_F^X \xrightarrow{_s\widetilde{F_{\Phi}}} \Gamma_K^{ab}\]
is independent of the choice of \(s\) and its restriction to \(\Gamma_{\mathbb{Q}}\) is \(F_{\Phi}\). 
\begin{proof} This is Proposition 2.1.3 and Proposition 2.1.7 of \cite{Nekovar}.
\end{proof}
\end{proposition}

Using the language developed in \S2, we can give an alternative description of Nekov\'a\v{r}'s half-transfer maps.  As with the classical case, we choose a section \(\rho \mapsto w_{\rho} \) to the projection \(\Gamma_{\mathbb{Q}}  \twoheadrightarrow \Hom_{\Qalg}(K , \overline{\mathbb{Q}})\) such that \(w_{c \rho} = cw_{\rho}\) for all \(\rho \in \Sigma_K\).  We notice that Tate's original definition makes sense for any element \(\alpha\) of the set  \(\Gamma_{\mathbb{Q}} \# \Gamma_{K} :=  \Aut_{\rsets \Gamma_{K}}\Gamma_{\mathbb{Q}}\) of all automorphisms of \(\Gamma_{\mathbb{Q}}\) as a right \(\Gamma_K\)-set. Indeed any such \(\alpha\) induces a well-defined permutation \(\rho \mapsto \alpha \rho\) of \(\Sigma_K\) and for any \(\alpha \in \Gamma_{\mathbb{Q}} \# \Gamma_{K} \) the elements 
\[ h_{\rho}(\alpha) = w_{\alpha \rho}^{-1}\alpha(w_{\rho}) \]
lie in the subgroup \( \Gamma_K \subset \Gamma_{\mathbb{Q}} \). For any \(\alpha \in \Gamma_{\mathbb{Q}} \# \Gamma_{K} \) we define an element  
\[ F_{\Phi}(\alpha) = \prod_{\phi \in \Phi} h_{\phi}(\alpha)  \;\; \text{ mod } \Gamma_{K^{ab}} \]
and it is clear that if we compose with the inclusion \(\Gamma_{\mathbb{Q}} \hookrightarrow \Gamma_{\mathbb{Q}} \# \Gamma_{K}\) we recover Tate's half-transfer map \(F_{\Phi}\).  

\begin{proposition} If \(\alpha \in \Gamma_{\mathbb{Q}} \# \Gamma_{F} \subset \Gamma_{\mathbb{Q}} \# \Gamma_{K} \) then \({F_{\Phi}}(\alpha) \in \Gamma_K^{ab} \) is independent of the choice of coset representatives \(\rho \mapsto w_{\rho}\).
\begin{proof} Denote by \( F_{\Phi}^{\prime}(\alpha)\) the element of \(\Gamma_K^{ab} \) we would have obtained had we chosen another collection \(\{w_{\rho}^{\prime}\}\) of coset representatives. Then since each \(w_{\rho}^{\prime}\) is of the form \(w_{\rho}^{\prime} = w_{\rho}e_{\rho}\) for some \(e_{\rho} \in \Gamma_K\) we see that
\( h_{\rho}^{\prime}(\alpha) = w^{\prime -1}_{\alpha \rho}\alpha(w^{\prime}_{\rho}) = e_{\alpha \rho}^{-1} h_{\rho}^{\prime}(\alpha) e_{\rho} \)
and so
 \[ \frac{{F_{\Phi}^{\prime}}(\alpha)}{F_{\Phi}(\alpha)} =  \prod_{\phi \in \Phi} e_{\alpha\phi}^{-1}e_{\phi}  \;\; \text{ mod } \Gamma_{K^{ab}} \]

 The conditions on \(w_{\rho}\) and \(w_{\rho}^{\prime}\) force \(e_{c\rho} = e_{\rho}\) and so \(e_{\rho}\) only depends on \(\rho|_{F}\). The condition \(\alpha \in \Gamma_{\mathbb{Q}} \# \Gamma_{F} \subset \Gamma_{\mathbb{Q}} \# \Gamma_{K} \) ensures that \(\alpha\Phi = \{\alpha \phi: \phi \in \Phi\}\) is another CM type for \(K\).  It follows that the product on the right hand side is trivial and \(F_{\Phi}(\alpha) = F_{\Phi}^{\prime}(\alpha) \) as claimed. 
\end{proof}
\end{proposition}

Exactly as with Tate's half-transfer maps, we have the following compatibilities:

\begin{proposition}  Let \(K \subset \overline{\mathbb{Q}}\) be a CM field and let \(\Phi \subset \Sigma_K =  \Hom_{\Qalg} (K, \mathbb{C})\) be a CM type.

  (i) Let \(K^{\prime} \subset \overline{\mathbb{Q}}\) be a CM field containing \(K\), \(F^{\prime} = K^{\prime \; +}\) and \(\Phi^{\prime}\) the CM type of \(K^{\prime}\) induced from \(\Phi\) on \(K\), i.e. \(
 \Phi^{\prime} = \{ \rho^{\prime} \in \Sigma_{K^{\prime}} : \rho^{\prime}|_K \in \Phi \} \).
The maps \(F_{\Phi}\) and \(F_{\Phi^{\prime}}\) fit into a commutative diagram:
\[ \begin{CD} \Gamma_{\mathbb{Q}} \# \Gamma_{F}  @> F_{\Phi} >>  \Gamma_K^{ab} \\
@VVV @VV {\Ver}_{K^{\prime}/K} V \\
\Gamma_{\mathbb{Q}} \# \Gamma_{F^{\prime}} @> F_{\Phi} >>   \Gamma_{K^{\prime}}^{ab} \end{CD} \] 
where \( {\Ver}_{K^{\prime}/K} : \Gamma_K^{ab} \rightarrow  \Gamma_{K^{\prime}}^{ab}\) denotes the transfer (Verlagerung) map coming from \(\Gamma_{K^{\prime}} \subset 
\Gamma_K\)

(ii) Let \(u \in \Gamma_{\mathbb{Q}}\), \(K^{\prime} = u(K) \subset \overline{\mathbb{Q}}\), \(F^{\prime} = K^{\prime \; +} = u(F)\) and  \(\Phi^{\prime} \subset \Sigma_{K^{\prime}}\) the CM type \(\Phi^{\prime} = \Phi u^{-1} = \{ \rho^{\prime} \in \Sigma_{K^{\prime}} : \rho^{\prime} \circ u|_K \in \Phi \} \). The maps \(F_{\Phi}\) and \(F_{\Phi^{\prime}}\) fit into a commutative diagram:
\[ \begin{CD}  \Gamma_{\mathbb{Q}} \# \Gamma_{F}  @> F_{\Phi} >>  \Gamma_K^{ab} \\
@V [u] VV @VV u_{\ast} V \\
\Gamma_{\mathbb{Q}} \# \Gamma_{F^{\prime}}  @>{F_{\Phi^{\prime}}} >>   \Gamma_{K^{\prime}}^{ab}  \end{CD} \] 
where \( u_{\ast} : \Gamma_K^{ab} \rightarrow \Gamma_{K^{\prime}}^{ab}  \) is the map induced by the map \(\Gamma_K \rightarrow \Gamma_{K^{\prime}}\) given by  \( g \mapsto u g  u^{-1} \).   

\begin{proof}  Both of these are straightforward calculations using the computations in Proposition 2.3 and Proposition 2.4 respectively. 
\end{proof}
\end{proposition}

In \S2 we defined a canonical isomorphism \( \Aut_F(F \otimes_{\mathbb{Q}} \overline{\mathbb{Q}}) \xrightarrow{\sim}   \Gamma_{\mathbb{Q}} \# \Gamma_{F}\). The following theorem tells us that Nekov\'a\v{r}'s half-transfer maps \(\widetilde{F_{\Phi}}\) correspond under this isomorphism to the very natural maps \(F_{\Phi}\) defined above. 

\begin{theorem} Let \(\alpha_g\) denote the image in \(\Gamma_{\mathbb{Q}} \# \Gamma_{F}\) of \(g \in \Aut_F(F \otimes_{\mathbb{Q}} \overline{\mathbb{Q}}) \). Then 
\( \widetilde{F_{\Phi}}(g) = F_{\Phi}(\alpha_g) \in \Gamma_K^{ab}\) for all \(g \in \Aut_F(F \otimes_{\mathbb{Q}} \overline{\mathbb{Q}}) \).  
 \begin{proof} Let us make choices as in the definition of \(\widetilde{F_{\Phi}} \). In particular we have chosen a section \(s = (s_x)_{x \in X}\) to the projection \(\Gamma_{\mathbb{Q}} \twoheadrightarrow X = \Sigma_F \) and this gives us a section \(\rho \mapsto w_{\rho} \) to \(\Gamma_{\mathbb{Q}} \twoheadrightarrow \Sigma_K = \Hom_{\Qalg}(K ,\overline{\mathbb{Q}}) \) such that \(w_{c\rho} = cw_{\rho}\) by setting
\( w_{c^b s_x|_K} = c^{b} s_x\). We identify the CM type \(\Phi\) with \(a = (a_x)_{x\in X} \in (\mathbb{Z}/ 2\mathbb{Z})^X\). 
To ease the notation write \(\alpha\) for \(\alpha_g\) and let us write \( \rho_s(\alpha) = (\pi, (h_x)_{x \in X}) = \beta_s(g)\).  Recall this means that \(\pi\) is the permutation of \(X = \Sigma_F\) induced by \(\alpha\) and we have 
\(h_x = s_{\pi(x)}^{-1}\alpha(s_x) \).
 By definition 
\[ F_{\Phi}(\alpha) = \prod_{\phi \in \Phi} w_{\alpha \phi}^{-1} \alpha(w_{\phi}) = \prod_{x \in X} w_{\alpha \phi_x}^{-1} \alpha(w_{\phi_x})  \;\; \text{ mod } \Gamma_{K^{ab}}\]
where \(\phi_x = c^{a_x}s_x|_K\) for all \(x \in X\). We can make the following computations
\begin{itemize} 
\item \(\alpha(w_{\phi_x}) = s_{\pi(x)}h_xs_x^{-1}c^{a_x}s_x \in \Gamma_{\mathbb{Q}} \)
\item \(w_{\alpha \phi_x} = c^{a_x + \overline{h_x}} s_{\pi(x)} \in \Gamma_{\mathbb{Q}} \) 
\end{itemize} 
To see this recall that by definition \(h_x = s_{\pi(x)}^{-1}\alpha(s_x)\) so the right hand side of the first claim is just
\[ s_{\pi(x)}h_xs_x^{-1}c^{a_x}s_x = \alpha(s_x)s_x^{-1}c^{a_x}s_x\] 
But then as \(K\) is CM we have \(s_x^{-1}c^{a_x}s_x|_K = c^{a_x}|_K\) and in particular \(s_x^{-1}c^{a_x}s_x|_F = 1\), i.e. \(s_x^{-1}c^{a_x}s_x \in \Gamma_F\). Since \(\alpha \in  \Gamma_{\mathbb{Q}} \# \Gamma_{F} \) we can take this inside the brackets to get
 \[ \alpha(s_xs_x^{-1}c^{a_x}s_x) = \alpha(c^{a_x}s_x) = \alpha(w_{\phi_x}) \] 
For the second claim, the fact that \(\alpha(w_{\phi_x}) = s_{\pi(x)}h_xs_x^{-1}c^{a_x}s_x\) implies that \(\alpha(w_{\phi_x}) |_K = s_{\pi(x)}h_xc^{a_x}|_K = c^{a_x +  \overline{h_x}} s_{\pi(x)} |_K\) (using again the fact that \(K\) is CM). This shows that \(w_{\alpha \phi_x} = c^{a_x + \overline{h_x}} s_{\pi(x)}|_K\) and so by definition \(w_{\alpha \phi_x} = c^{a_x + \overline{h_x}} s_{\pi(x)}\). 

With this established, we compute straight from the definition that
\[F_{\Phi}(\alpha) = \prod_{x \in X} w_{\alpha \phi_x}^{-1} \alpha(w_{\phi_x})   \;\; \text{ mod } \Gamma_{K^{ab}}  = \prod_{x \in X} s_{\pi(x)}^{-1} c^{a_x + \overline{h_x}}s_{\pi(x)}h_xs_x^{-1}c^{a_x}s_x|_{K^{ab}} = \widetilde{F_{\Phi}}(g)   \]
as claimed. 
\end{proof}
\end{theorem}

\section{The plectic Serre group} 

Let \(E\subset \overline{\mathbb{Q}}\) be any CM field and for each \(\sigma \in \Hom_{\Qalg}(E, \overline{\mathbb{Q}})\) let \([\sigma]\) denote the associated character of the torus \(T_E = \Res_{E/\mathbb{Q}} \mathbb{G}_m\). These characters form a \(\mathbb{Z}\)-basis for the character group \(X^{\ast}(T_E)\) and the arithmetic action of \(\tau \in \Gal(\overline{\mathbb{Q}} / {\mathbb{Q}})\) is given in terms of this basis by \([\sigma] \mapsto [\tau \sigma]\). The level \(E\) Serre group is the quotient  \(\mathcal{S}^E\) of \(T_E = \Res_{E/\mathbb{Q}} \mathbb{G}_m\) with character group \[ X^{\ast}(\mathcal{S}^E) = \{ \sum_{\sigma} n_{\sigma}[\sigma] \; : \; \text{there is some } w \in \mathbb{Z} \mbox{ such that } n_{\sigma} + n_{c\sigma} = w \text{ for all } \sigma \;   \} \subset X^{\ast}(T_E)  \]
When \(E\) is Galois over \(\mathbb{Q}\) the action of the Galois group \(\Gamma_{E/\mathbb{Q}}\) on \(E^{\times}\) gives rise to an algebraic action \(\Gamma_{E/\mathbb{Q}} \hookrightarrow \Aut({T}_E)\). The algebraic action of an element \(\tau \in \Gamma_{E/\mathbb{Q}}\) can be described in terms of characters by \([\sigma] \mapsto [ \sigma \circ \tau ]\) and induces a well-defined action of \(\Gamma_{E/\mathbb{Q}}\) on the quotient \(\mathcal{S}^E\). As discussed in the introduction, for such \(E\) Langlands' has constructed a particular extension 
 \[ 1 \rightarrow \mathcal{S}^E \rightarrow \mathcal{T}^E \rightarrow \Gamma_{E^{ab} / \mathbb{Q}} \rightarrow 1\] 
 of the profinite group \(\Gamma_{E^{ab} / \mathbb{Q}}\) by the torus \(\mathcal{S}^E\). The action of \(\Gamma_{E^{ab} / \mathbb{Q}}\) on \(\mathcal{S}^E\) implicit in this level \(E
\) Taniyama extension is given by precomposing this algebraic action of \(\Gamma_{E/\mathbb{Q}}\) defined above with the natural projection \(\Gamma_{E^{ab} / \mathbb{Q}} \twoheadrightarrow \Gamma_{E / \mathbb{Q}}\).

Now let \(F\) be any totally real number field and suppose \(E\) is a CM number field which is big enough to split \(F\). In this case we can identify the set \(\Sigma_F =  \Hom_{\Qalg}(F,\overline{\mathbb{Q}})\) with \(\Hom_{\Qalg}(F,E)\). 

  \begin{definition}  The \(F\)-plectic Serre group at level \(E\) is the product group \(\mathcal{S}^E_F := (\mathcal{S}^E)^{\Sigma_F}\), i.e. the quotient of \(T_E^{\Sigma_F}\) with character group 
 \[  X^{\ast}(\mathcal{S}_F^E) = \{ \sum_{\sigma : E\hookrightarrow \overline{\mathbb{Q}}} n_{j, \sigma}[j, \sigma] \; : \; \text{there is some } \underline{w} = (w_j) \in \mathbb{Z}^{\Sigma_F} \mbox{ such that } n_{j, \sigma} + n_{j, c\sigma} = w_j \;   \}  \]
 where \([j, \sigma]\) denotes the character of \(T_E^{\Sigma_F}\) obtained by first projecting onto the \(j\) component and then applying the character \([\sigma]\) of \(T_E\). 
 \end{definition}
 
 When \(E\) is Galois over \(\mathbb{Q}\), we can define an action of the \(F\)-plectic Galois group \(\Aut_F(F\otimes_{\mathbb{Q}}E)\) on \(\mathcal{S}_F^E\). First note that the natural isomorphism 
\(  F \otimes_{\mathbb{Q}} E \xrightarrow{\sim} E^{\Sigma_F} \) of \(\mathbb{Q}\)-algebras
given by \(f \otimes e \mapsto (j(f)e)_{j \in {\Sigma_F}}\) induces an isomorphism of \(\mathbb{Q}\)-tori
 \( T_{F \otimes_{\mathbb{Q}} E} := \Res_{F \otimes_{\mathbb{Q}} E/ \mathbb{Q}} \mathbb{G}_m \xrightarrow{\sim} T_E^{\Sigma_F} \), and in this way we can view \(\mathcal{S}_F^E\) as a quotient of \(T_{F \otimes_{\mathbb{Q}} E} \).  We have an isomorphism 
\[(F \otimes_{\mathbb{Q}} E \otimes_{\mathbb{Q}} \overline{\mathbb{Q}} )^{\times} \xrightarrow{\sim} \prod_{j \otimes \sigma} \overline{ \mathbb{Q}}^{\times} \]
where here \(j \otimes \sigma\) signifies that the action of \(F\) is via \(j : F \hookrightarrow \overline{\mathbb{Q}} \) and then action of \(E\) is via \(\sigma : E \hookrightarrow  \overline{\mathbb{Q}} \). If we let \([j \otimes \sigma]\) denote the projection onto the \(j \otimes \sigma)\) component of the above decomposition then the arithmetic action of \(\tau \in \Gal(\overline{\mathbb{Q}} / {\mathbb{Q}})\) is given in terms of this character basis by \([j \otimes \sigma] \mapsto [\tau \circ j \otimes \tau \circ \sigma]\). In this way we identify the characters \([j,\sigma]\) of \(T_E^{\Sigma_F}\) and \([\sigma \circ j \otimes \sigma]\) of \(T_{F \otimes_{\mathbb{Q}} E}\).

 Now we make use of the language developed in \S2. Recall that an embedding \(j : F \hookrightarrow E\) defines a map \(j :  \Aut_F(F\otimes_{\mathbb{Q}} E) \hookrightarrow \Aut_E(E\otimes_{\mathbb{Q}} E) = S_{\Gamma_{E/\mathbb{Q}}}\). For \(g \in \Aut_{F}(F \otimes_{\mathbb{Q}} E)\) define a homomorphism 
\( g^{\ast} : X^{\ast}(T_{F\otimes_{\mathbb{Q}}E}) \rightarrow X^{\ast}(T_{F\otimes_{\mathbb{Q}}E}) \)
by 
\( [ j \otimes \sigma] \mapsto [j \otimes \sigma j(g)] \)
where \[\sigma j(g) := (j(g)^{-1}(\sigma^{-1}))^{-1} \in \Gamma_{E/\mathbb{Q}} \] We will try to be careful to to distinguish this from the element \(\sigma \circ j (g) \in S_{\Gamma_{E/\mathbb{Q}}} \). That this is equivariant for the arithmetic action of \(\Gamma_{\mathbb{Q}}\) on \(X^{\ast}(T_{F\otimes_{\mathbb{Q}}E})\) follows from Theorem 2.5 which implies that if \(j^{\prime} = \tau \circ j \) then \(j^{\prime}(g)(\sigma) = j(g)(\sigma \tau)\cdot \tau^{-1}\). 

It's easy to check this defines an action of \(\Aut_{F}(F \otimes_{\mathbb{Q}} E) \) on \(T_{F\otimes_{\mathbb{Q}}E}\) (and hence on  \(T_{E}^{\Sigma_F}\)), which we denote on points by \(t \mapsto g \star t\) so that on \(\overline{\mathbb{Q}}\)-points we have \( g^{\ast}\chi(t) = \chi(g \star t) \). The following proposition implies that this action stabilises \(X^{\ast}(\mathcal{S}_F^E) \subset  X^{\ast}(T_{E}^{\Sigma_F}) = X^{\ast}(T_{F\otimes_{\mathbb{Q}}E})\) and thus induces a well-defined algebraic action of \(\Aut_F(F\otimes_{\mathbb{Q}} E)\)  on \(\mathcal{S}_F^E\). 
 

\begin{proposition} Let \(j_0 : F \hookrightarrow E\) and let \(\chi\) be a character of \(T_{E}^{\Sigma_F}\) pulled back from the character of \(T_{E}\) attached to a CM type \(\Phi \subset \Gamma_{E/\mathbb{Q}} \) via the \(j_0^{th}\) projection \(T_{E}^{\Sigma_F} \twoheadrightarrow T_E\). Then if \(g \in \Aut_{F}(F \otimes_{\mathbb{Q}} E)\) maps \(j_0\) to \(j_0^{\prime} : F \hookrightarrow E\) then the character \( g^{\ast} \chi \in X^{\ast}(T_{E}^{\Sigma_F})\) is the pull back of a character of \(T_{E}\) attached to a (possibly different) CM type \(\Phi^{\prime} \subset \Gamma_{E/\mathbb{Q}}\) via the \(j_0^{\prime \; th}\) projection \(T_{E}^{\Sigma_F} \twoheadrightarrow T_E\).
\begin{proof} Let's write \(\chi = \sum_{j, \sigma} a_{j,\sigma} [j, \sigma] = \sum_{j, \sigma} b_{j \otimes \sigma}[ j \otimes \sigma ] \) where \(b_{j \otimes \sigma} = a_{\sigma^{-1}j, \sigma}\).  We will also write 
 \(g^{\ast} \chi = \sum_{j, \sigma} c_{j,\sigma} [j, \sigma] = d_{j \otimes \sigma}[ j \otimes \sigma ] \)
 and we are considering the case where \(a_{j, \sigma} \in \{0,1\}\) and \(a_{j, \sigma} = 1\) if and only if we have both \(j = j_0\) and \(\sigma \in \Phi = \{\sigma_1, \dots, \sigma_d\}\) for a CM type \(\Phi \subset \Gamma_{E/\mathbb{Q}} \).  Pictorially we can represent this by a matrix
\[\bordermatrix{
~ & j_0 & &  j_1  & & j_d  \cr
\sigma_1 &  1 & \dots & 0 & \dots & 0\cr
\vdots   & \vdots & \dots  & 0 & \dots  & 0 \cr
\sigma_d  & 1  & \dots  & 0 & \dots &  0\cr
c \sigma_1  &0 & \dots &  0 &  \dots & 0\cr
  & \vdots &  \dots & 0  &\ddots & 0 \cr
c\sigma_d   &0 &\dots & 0 & \dots  & 0
 }_{j, \sigma} = \bordermatrix{
~ & j_0 & &  j_1  & & j_d  \cr
\sigma_1 &  0 & \dots & 1 & \dots & 0\cr
\vdots   & \vdots & \dots  & 0 & \dots  & 0 \cr
\sigma_d  & 0  & \dots  & 0 & \dots &  1\cr
c \sigma_1   &0 & \dots &  0 &  \dots & 0\cr
  & \vdots &  \dots & 0  &\ddots & 0 \cr
c \sigma_d  &0 &\dots & 0 & \dots  & 0
 }_{j \otimes \sigma} \]
 where here for each \(r = 1, \dots, d\) we let \(j_r = \sigma_r \circ j_0\). We can translate between the two matrices by an appropriate permutation of each row (i.e. \(j \mapsto \sigma \circ j\) in row \(\sigma\)).   For any \(g \in \Aut_F(F\otimes_{\mathbb{Q}} E)\) the permutation \(j_r(g)\) lies in the subgroup \(  \Gamma_{E / \mathbb{Q}} \# \Gamma_{E/j_r(F)} \subset S_{ \Gamma_{E/\mathbb{Q}}}\) 
and so maps 
 \[ \{ \tau \in \Gamma_{E/\mathbb{Q}}  : \tau \circ j_r = j_0 \} \mapsto  \{ \tau \in \Gamma_{E/\mathbb{Q}}  : \tau \circ j_r = j_0^{\prime}\} \]
 where \(j_0^{\prime}= gj_0\). Hence the right action defined by \(\sigma \mapsto (j_r(g)^{-1} (\sigma^{-1}))^{-1}\) sends 
  \[ \{ \sigma \in \Gamma_{E/\mathbb{Q}} : \sigma \circ j_0= j_r\} \mapsto  \{ \sigma \in \Gamma_{E/\mathbb{Q}}  : \sigma \circ j_0^{\prime} = j_r\} \]
  Since the action of \(g\) on \(X^{\ast}(T_{F\otimes E})\) acts on the \(( j_r \otimes -)\) column above by the right action of \(j_r(g)\) we deduce that  \(g^{\ast} \chi\) is of the form 
  \[\bordermatrix{
~ & & &  j_0^{\prime} & &   \cr
\sigma_1 &  0 & \dots & \ast & \dots & 0\cr
\vdots   & 0 & \dots & \ast & \dots & 0\cr
\sigma_d  & 0 & \dots & \ast & \dots & 0\cr
c\sigma_1  &0 & \dots & \ast & \dots & 0\cr
  & 0 & \dots & \ast & \ddots & 0\cr
c\sigma_d   & 0 & \dots & \ast & \dots & 0\cr
 }_{j, \sigma}\]
 More precisely we have shown that for \(\sigma \in \Gamma_{E/\mathbb{Q}} \) such that \(\sigma \circ j_0^{\prime} = j\) we have 
 \(c_{j_0^{\prime}, \sigma} = a_{j_0, \tau} \)
 where \(\tau \in \Gamma_{E/\mathbb{Q}} \) is such that \(\tau j(g) = \sigma\). Since \(F\) is totally real, this implies that  \(\Phi^{\prime} = \{\sigma \in \Gamma_{E/\mathbb{Q}}  : c_{j_0^{\prime}, \sigma} = 1\}\) is a CM type of \(E\). 
  \end{proof}
 \end{proposition}


\section{The plectic universal Taniyama element}  

Let \(F \subset \overline{\mathbb{Q}}\) denote a totally real number field and \(E \subset \overline{\mathbb{Q}}\) a CM field which is Galois over \(\mathbb{Q}\) and large enough to contain \(F\). In this section we ape Langlands' construction of the universal Taniyama element to produce a map
\[f_F^E :  \Aut_F(F\otimes E^{ab} ) \rightarrow \mathcal{S}_F^E(\mathbb{A}_{E,f})/\mathcal{S}_F^E({E}) \]
which we call the \(F\)-plectic universal Taniyama element. 

Let \(W_{E/\mathbb{Q} , f}\) denote the finite-adelic version of the Weil group of the extension \(E / \mathbb{Q}\), which sits in a diagram
  \[ \begin{CD} 1 @>>> \mathbb{A}_{E,f}^{\times} / E^{\times} @ >>> W_{E/\mathbb{Q}, f} @>>> \Gamma_{E/\mathbb{Q}} @>>> 1 \\
  & & @Vart_EVV @V \varphi_f VV @V \text{ id } VV & \\
  1 @>>> \Gamma_E^{ab} @>>> \Gamma_{E^{ab} / \mathbb{Q}} @>>> \Gamma_{E/\mathbb{Q}}  @>>> 1 \end{CD} \]
in which the columns are surjective. The implicit action of \(\sigma \in \Gamma_{E/\mathbb{Q}}\) on \(\mathbb{A}_{E, f}^{\times} / E^{\times}\) is via \(e \mapsto e^{\sigma}\). We need to make some choices which will turn out to play no role:

\begin{itemize}
\item  Corresponding to the place \(v\) of \(\overline{\mathbb{Q}}\) coming from the embedding into \(\mathbb{C}\) (which we have fixed) there is a continuous map \(i_v : W_{\mathbb{R}} \rightarrow W_{\mathbb{Q}} \) (uniquely defined up to an inner automorphism of \(W_{\mathbb{Q}}\) by an element of \(\Ker( \varphi)\)) and a diagram
 \[ \begin{CD} W_{\mathbb{R}} @> \varphi_v >> \Gal(\mathbb{C}/\mathbb{R}) \\
 @V i_v VV @VVV \\
 W_{\mathbb{Q}} @> \varphi >>  \Gamma_{\mathbb{Q}} \end{CD} \]
 inducing the inclusion \(\mathbb{C}^{\times} \hookrightarrow C_E \simeq W_{E}^{ab} \) coming from the place \(v\). The map \(i_v\) gives a well defined lift of complex conjugation \(c \in \Gamma_{E^{ab} / \mathbb{Q}}\) to an element \(\tilde{c} \in  W_{E/\mathbb{Q}, f}\) of order 2 (which changes by conjugation by an element of \(\Ker(\varphi_f)\) if we change \(i_v\) - see \cite{Tate}, 1.6.1). We choose one such map \(i_v\) and let \(\tilde{c}\)  denote the associated complex conjugation. 
 \item Having fixed \(i_v\) we then pick elements \({w_{\sigma}} \in W_{E/\mathbb{Q}, f} \) such that  \(\sigma \mapsto w_{\sigma}\) is a section to the surjection \(W_{E/\mathbb{Q}, f} \twoheadrightarrow \Gamma_{E/\mathbb{Q}}\) and further we have the relation \(w_{c\sigma} = \tilde{c} w_{\sigma} \) for all \(\sigma \in \Gamma_{E/\mathbb{Q}}\) 
\end{itemize}

To construct the \(F\)-plectic universal Taniyama element we exploit the canonical isomorphism \(\Aut_F(F\otimes_{\mathbb{Q}} E^{ab}) \xrightarrow{\sim}  \Gamma_{E^{ab}/\mathbb{Q}} \# \Gamma_{E^{ab}/F}\). An embedding \(j : F \hookrightarrow E\) induces a natural map \( j : \Aut_F(F\otimes E^{ab}) \hookrightarrow \Aut_E(E \otimes E^{ab})\) which, in terms of plectic groups, is now precisely the composite map  \[ \Gamma_{E^{ab}/\mathbb{Q}} \# \Gamma_{E^{ab}/F} \xrightarrow{[{u}]} \Gamma_{E^{ab}/\mathbb{Q}} \# \Gamma_{E^{ab}/ j(F) } \subset \Gamma_{E^{ab}/\mathbb{Q}} \# \Gamma_{E^{ab}/ E }\] where here \(u \in \Gamma_{E^{ab}/\mathbb{Q}}\) is any element with \(u|_F = j\). We have a similar map on Weil groups and if we decompose \[{T}_{F \otimes E} (\mathbb{A}_{E,f})  / {T}_{F \otimes E} ({E})  \xrightarrow{\sim} \prod_{ j \otimes \sigma}  \mathbb{A}_{E,f}^{\times} / E^{\times} \] with \(j \in \Sigma_F\) and \(\sigma \in \Gamma_{E/\mathbb{Q}} \) we can define the \(F\)-plectic universal Taniyama element as follows:

\begin{definition} For \(\alpha \in \Gamma_{E^{ab}/ \mathbb{Q}} \# \Gamma_{E^{ab}/ F}\) we define
\(f_{F}^{E}(\alpha)  \in \mathcal{S}_F^E (\mathbb{A}_{E,f})  / \mathcal{S}_F^E({E}) \) 
to be the image of \[ h_F^E(\tilde{\alpha}) := (h_{\sigma^{-1}}(j(\tilde{\alpha})))_{j \otimes \sigma} := (w_{ j(\tilde{\alpha})\sigma^{-1}}^{-1} j(\tilde{\alpha})(w_{\sigma^{-1}}))_{j \otimes \sigma}    \in   {T}_{F\otimes E} (\mathbb{A}_{E,f})  / {T}_{F\otimes E} ({E})   \]
under \(T_{F \otimes E} \twoheadrightarrow \mathcal{S}_F^E\), where \(\tilde{\alpha}\) is any lift of \(\alpha\) to an element of \(W_{E/\mathbb{Q} , f} \# W_{E/F , f}\). For \(g  \in  \Aut_F(F\otimes E^{ab} )\) we set \[f_F^E(g) : = f_F^E(\alpha_g) \in \mathcal{S}_F^E(\mathbb{A}_{E,f})/\mathcal{S}_F^E({E})\] where \(\alpha_g \in \Gamma_{E^{ab}/\mathbb{Q}} \# \Gamma_{E^{ab}/F}\) is the image of \(g\) under the canonical isomorphism  \(\Aut_F(F\otimes_{\mathbb{Q}} E^{ab}) \xrightarrow{\sim}  \Gamma_{E^{ab}/\mathbb{Q}} \# \Gamma_{E^{ab}/F}\).
\end{definition}

Here \(W_{E/F, f}\) denotes the subgroup of \(W_{E/\mathbb{Q}, f}\) consisting of elements whose image in \(\Gamma_{E/\mathbb{Q}}\) lies in \(\Gamma_{E/F}\). The surjection \(\varphi_f :    W_{E/\mathbb{Q} , f} \twoheadrightarrow \Gamma_{E^{ab}/\mathbb{Q}}\) induces a surjection \(\varphi_f :  W_{E/\mathbb{Q} , f} \# W_{E/F , f} \twoheadrightarrow \Gamma_{E^{ab}/ \mathbb{Q}} \# \Gamma_{E^{ab}/ F}\). Each \(j(\tilde{\alpha}_g)\) is a element of \(W_{E/\mathbb{Q} , f} \# \mathbb{A}_{E,f}^{\times}/E^{\times}\) and so for any \(\sigma \in \Gamma_{E/\mathbb{Q}}\) the difference \(w_{ j(\tilde{\alpha}_g)\sigma^{-1}}^{-1} \cdot j(\tilde{\alpha}_g)(w_{\sigma^{-1}})\) lies in \(\mathbb{A}_{E,f}^{\times}/E^{\times}\). Then \(h_F^E( \tilde{\alpha}_g)\)
simply stores this collection of elements of  \(\mathbb{A}_{E,f}^{\times}/E^{\times}\) for each \(j \in \Sigma_F\) and \(\sigma \in \Gamma_{E/\mathbb{Q}} \) in an appropriate way.

\begin{remark} Let \(K\) be a CM field with maximal totally real subfield \(F\). Then if \(E\) contains \(K\) we have a representation 
 \[\rho_{\Phi, F} : T_{F\otimes_\mathbb{Q}E}  \rightarrow T_K\] from \S1 whose restriction to the diagonal is the classical reflex norm but which is general will not factor through the quotient \(T_{F\otimes_\mathbb{Q}E} \twoheadrightarrow \mathcal{S}_F^E\). This representation can be described on character groups by the map \[ X^{\ast}(T_K)   \ni [i] \mapsto  \sum_{j, \sigma } \Phi_{j}(\sigma^{-1} \circ i)[j, \sigma] = \sum_{j, \sigma } \Phi_{\sigma^{-1}\circ j}(\sigma^{-1} \circ i)[j \otimes \sigma]  \in X^{\ast}(T_{F\otimes_\mathbb{Q}E} ) \]
 for any \(i: K \hookrightarrow E\). 
The element 
 \[ \rho_{\Phi, F}(h_F^E(\tilde{\alpha}_g)) \in  {T}_{K} (\mathbb{A}_{E,f})  / {T}_{K} ({E}) = (K \otimes_{\mathbb{Q}} \mathbb{A}_{E,f}^{\times} )/ (K \otimes_{\mathbb{Q}} E)^{\times}  \xrightarrow{\sim} \prod_{i : K \hookrightarrow E} \mathbb{A}_{E,f}^{\times}/E^{\times}   \] 
might depend both on the choice of lift \(\tilde{\alpha}_g\) of \(\alpha_g\) and also on the choices we made in the construction, but can nevertheless be related to Nekovar's plectic-half transfer maps. More precisely, together Proposition 3.3 and Theorem 3.4 imply that for any \(i: K \hookrightarrow E\) the image of the \(i\) component of \(\rho_{\Phi, F}(h_F^E(\tilde{\alpha}_g)) \in  {T}_{K} (\mathbb{A}_{E,f})  / {T}_{K} ({E})\) under the Artin reciprocity map \(\mathbb{A}_{E,f}^{\times} / E^{\times} \twoheadrightarrow \Gamma_E^{ab}\) is always equal to the image of \(\widetilde{F_{\Phi}}(g) \in \Gamma_K^{ab}\) under the map \(\Gamma_K^{ab} \rightarrow \Gamma_E^{ab}\) induced by \(i\).
\end{remark}

Our next theorem justifies the choice of notation by showing that the \(F\)-plectic universal Taniyama element becomes well-defined independent of choices when we pass to the quotient \(T_{F \otimes E} \twoheadrightarrow \mathcal{S}_F^E\). The proof occupies most of the rest of this section, but is at heart a relatively straightforward generalisation of the calculations presented in \cite{Milne-Shih} for Langlands' Taniyama element.

\begin{theorem} For any \(g \in \Aut_F(F\otimes E^{ab} )\) the element \(f_F^E(g)\) is well-defined (i.e. independent of the choice of lift \(\tilde{\alpha}_g\)) and does not depend on the choice of \(i_v\) and of coset representatives \( w_{\sigma}\) for \(\sigma \in \Gamma_{E/\mathbb{Q}}\).
\begin{proof} Let us first fix a map \(i_v\) as above, and let \(\tilde{c} \in W_{E/\mathbb{Q}, f}  \) denote the associated lift of complex conjugation.  For \(\tilde{\alpha}  \in W_{E/\mathbb{Q} , f} \# W_{E/F , f} \) lifting \(\alpha \in \Gamma_{E^{ab}/\mathbb{Q}} \# \Gamma_{E^{ab}/ F}\) let's write  \[ h_F^E(\tilde{\alpha}) := (h_{\sigma^{-1}}(j(\tilde{\alpha})))_{j \otimes \sigma} := (w_{ j(\tilde{\alpha})\sigma^{-1}}^{-1} \cdot j(\tilde{\alpha})(w_{\sigma^{-1}}))_{j \otimes \sigma}    \in   {T}_{F\otimes E} (\mathbb{A}_{E,f})  / {T}_{F\otimes E} ({E})   \]
and \(f_F^E(\tilde{\alpha})\) for its image under \(T_{F \otimes E} \twoheadrightarrow \mathcal{S}_F^E\). We need to show that \(f_F^E(\tilde{\alpha})\) only depends on \(\alpha\) and not on \(\tilde{\alpha}\) or the choices of \(i_v\) and \( w_{\sigma}\) for \(\sigma \in \Gamma_{E/\mathbb{Q}}\). First we can show:

\begin{proposition}  For \(\tilde{\alpha}  \in W_{E/\mathbb{Q} , f} \# W_{E/F , f} \) the element  \[f_{F}^{E} (\tilde{\alpha}) \in  \mathcal{S}_F^E (\mathbb{A}_{E,f})  / \mathcal{S}_F^E(E)\]
does not depend of the choice of coset representatives \(w_{\sigma}\) satisfying such that \(w_{c\sigma} = \tilde{c} w_{\sigma} \) for all \(\sigma \in \Gamma_{E/\mathbb{Q}}\). 

\begin{proof} Since \(\mathcal{S}_F^E\) is split over \(E\) it is enough to show that for all \(\chi  \in X^{\ast}(\mathcal{S}_F^E) = \Hom({\mathcal{S}_F^E}_{/E}, {\mathbb{G}_m}_{/E})\) the element
\( \chi ( f_{F}^E (\tilde{\alpha})) \in  \mathbb{A}_{E,f}^{\times}/E^{\times}\)
does not depend on the choice of \(w = (w_{\sigma})_{\sigma \in \Gamma_{E/\mathbb{Q}}}\). Let \(\chi\) be a character of \( \mathcal{S}_F^E \) and write 
 \[\chi = \sum_{j, \sigma} a_{j,\sigma} [j, \sigma] = \sum_{j, \sigma} b_{j \otimes \sigma}[ j \otimes \sigma ] \] where \(b_{\sigma \circ j \otimes \sigma} = a_{j, \sigma}\), i.e. \(b_{j \otimes \sigma} = a_{\sigma^{-1} \circ j, \sigma}\). From \S4, we have an algebraic action of \(W_{E/\mathbb{Q} , f} \# W_{E/F , f} \) on \(\mathcal{S}_F^E\) via the quotient \(\Gamma_{E/\mathbb{Q}}\# \Gamma_{E/F} \), and let \(\chi^{\prime} \) denote the character \( ({\tilde{\alpha}^{-1}})^{\ast} \chi\) of \(\mathcal{S}_F^E\). If we write 
 \[ \chi^{\prime} = \sum_{j, \sigma} c_{j,\sigma} [j, \sigma]  = \sum_{j, \sigma} d_{j \otimes \sigma}[ j \otimes \sigma ] \]
then by the definition of the algebraic action we have  \(d_{j \otimes (j(\tilde{\alpha}) \sigma^{-1})^{-1}}  = b_{j \otimes \sigma}\), i.e. \(d_{j \otimes \sigma} = b_{j \otimes (j(\tilde{\alpha})^{-1}\sigma^{-1})^{-1}}\).
  If \( w = (w_{\sigma})\) is replaced by \(w^{\prime} = (w_{\sigma}^{\prime})\) with \(w_{\sigma}^{\prime} = w_{\sigma}e_{\sigma} \) and \(e_{\sigma} \in \mathbb{A}_{E,f}^{\times} /E^{\times} \) then we have 
\[ h_{\sigma}^{\prime}(j(\tilde{\alpha})) = (w^{\prime}_{j(\tilde{\alpha}) \sigma})^{-1} j(\tilde{\alpha})(w^{\prime}_{\sigma}) =  e_{j(\tilde{\alpha})\sigma}^{-1} e_{\sigma} h_{\sigma}(j(\tilde{\alpha})) \]
and we note here that the conditions on \(w_{\sigma}\) and \(w_{\sigma}^{\prime}\) (with respect to \(\tilde{c} \)) mean that \(e_{\sigma}\) only depends on \(\sigma|_{E^{+}}\), where \(E^+\) denotes the maximal totally real subfield of \(E\). 
Thus \( \chi(f_{F}^E (\tilde{\alpha})) \in \mathbb{A}_{E,f}^{\times} /E^{\times}\) changes by 
\[ \prod_{j, \sigma} \frac{ h_{\sigma^{-1}}^{\prime}(j(\tilde{\alpha}))^{b_{j \otimes \sigma}}}{{{h_{\sigma^{-1}}(j(\tilde{\alpha}))}^{b_{j \otimes \sigma}}}} =  \prod_{j, \sigma} \frac{ h_{\sigma}^{\prime}(j(\tilde{\alpha}))^{b_{j \otimes \sigma^{-1}}}}{{{h_{\sigma}(j(\tilde{\alpha}))}^{b_{j \otimes \sigma^{-1}}}}} = \prod_{j, \sigma} \frac{ e_{\sigma}^{b_{j \otimes \sigma^{-1}}} }{  e_{j(\tilde{\alpha})\sigma}^{b_{j \otimes \sigma^{-1}}}} =  \prod_{j, \sigma} \frac{ e_{\sigma}^{b_{j \otimes \sigma^{-1}}} }{  e_{\sigma}^{b^{\prime}_{j \otimes \sigma^{-1}}}} \in \mathbb{A}_{E,f}^{\times} /E^{\times}\]
where \(b^{\prime}_{j \otimes \sigma^{-1}} = b_{j \otimes (j(\tilde{\alpha})^{-1} \sigma )^{-1}} = d_{j \otimes \sigma^{-1}} \). Hence this difference is 
\( \prod_{\sigma} \frac{ e_{\sigma}^{b_{\sigma^{-1}}}}{e_{\sigma}^{d_{\sigma^{-1}}}} \)
where for each \(\sigma \in \Gamma_{E/\mathbb{Q}}\) we have \(b_{\sigma} = \sum_{j} b_{j \otimes \sigma} = \sum_j a_{j, \sigma} \) and \(d_{\sigma} = \sum_{j} d_{j \otimes \sigma} = \sum_j c_{j, \sigma} \).

We want to show that this product is trivial for all \(\chi\). In order to prove the claim we may assume \(\chi\) is the pull back of a character of \(\mathcal{S}^{E}\) attached to a CM type \(\Phi \subset \Gamma_{E/\mathbb{Q}} \) via the projection \(\mathcal{S}_F^E = (\mathcal{S}^E)^{\Sigma_F} \twoheadrightarrow T_E\) coming from an embedding \(j_0 : F \hookrightarrow E\). In other words, we have all \(a_{j, \sigma} \in \{0,1\}\) and \(a_{j, \sigma}\) is always zero if \(j \neq j_0\) and we have \(a_{j_0, \sigma} + a_{j_0, c\sigma} = 1\) for all \(\sigma \in \Gamma_{E/\mathbb{Q}}\). Thanks to Proposition 4.1, we know that in this case if \(j_0^{\prime} : F \hookrightarrow E\) is such that \(\tilde{\alpha} j_0^{\prime} = j_0\) then  we have \(c_{j, \sigma} = 0\) for all \(j \neq j_0^{\prime}\) and \(c_{j_0^{\prime}, \sigma} + c_{j_0^{\prime}, c\sigma} = 1\) for all \(\sigma \in  \Gamma_{E/\mathbb{Q}}\).  Thus if we let \(\Phi\) and \(\Phi^{\prime}\) denote the CM types attached to \(\chi\) and \(\chi^{\prime} = (\tilde{\alpha}^{-1})^{\ast} \chi \) we have \[b_{\sigma} = \sum_j a_{j, \sigma} = \Phi(\sigma)  \in \{0,1\}\] and similarly 
\(d_{\sigma} = \Phi^{\prime}(\sigma)\). We deduce that \( \chi ( f_{F}^E (\tilde{\alpha}))\) changes by
\[ \prod_{\sigma} \frac{ e_{\sigma}^{ \Phi(\sigma^{-1})  }}{e_{\sigma}^{\Phi^{\prime}(\sigma^{-1})}} \in \mathbb{A}_{E,f}^{\times} /E^{\times}\]
and this is just 1 since both \(\Phi\) and \(\Phi^{\prime}\) are CM types of \(E\) and \(e_{\sigma}\) only depends on \(\sigma|_{E^{+}}\). 

 \end{proof}
\end{proposition} 

From Proposition 5.2, we can deduce:

\begin{proposition} For any \(\tilde{\alpha} \in W_{E/\mathbb{Q} , f} \# W_{E/F , f}\) the element 
\(f_{F}^{E}(\tilde{\alpha}) \) lies in \(  \left[ \mathcal{S}_F^E(\mathbb{A}_{E,f} ) / \mathcal{S}_F^E(E) \right] ^{\Gal(E/\mathbb{Q})} \subset \mathcal{S}_F^E(\mathbb{A}_{E,f} ) / \mathcal{S}_F^E(E)\)

\begin{proof} Here we are consider the arithmetic action of \(\Gal(E/\mathbb{Q})\) coming from the action on \(\mathbb{A}_{E,f}\) points. Let \(\tau\) be in  \(\Gal(E/\mathbb{Q})\). If we write \({T}_E^{\Sigma_F} (\mathbb{A}_{E,f})  / {T}^{\Sigma_F}_E ({E})  \xrightarrow{\sim} \prod_{j, \sigma}  \mathbb{A}_{E,f}^{\times} / E^{\times}  \) then the arithmetic action of \(\tau\) on the left hand side translates into the action \((e_{j, \sigma})_{j, \sigma} \mapsto (e^{\tau}_{j, \tau^{-1}\sigma})_{j, \sigma}\) on the right hand side.  By definition \(f_{F}^E(\tilde{\alpha} ) \) is equal to the image of \((h_{\sigma^{-1}}(j(\tilde{\alpha} )))_{j \otimes \sigma} =  (h_{\sigma^{-1}}( \sigma \circ j(\tilde{\alpha} )))_{j, \sigma}\) under \(T_E^{\Sigma_F} \twoheadrightarrow \mathcal{S}_F^E\) and hence \(\tau^{-1}(f_{F}^E(\tilde{\alpha} ))\) is the image of \[ (h^{\tau^{-1}}_{\sigma^{-1}\tau^{-1}}(\tau \sigma \circ j(\tilde{\alpha} )))_{j, \sigma}  = (\tilde{\tau}^{-1} \cdot w^{-1}_{ \tau \sigma \circ j(\tilde{\alpha} )( \sigma^{-1} \tau^{-1})}  \cdot \tau \sigma \circ j(\tilde{\alpha} )  (w_{\sigma^{-1} \tau^{-1}}) \cdot \tilde{\tau} )_{j, \sigma}\]
where \(\tilde{\tau}\) is any lift of \(\tau\) to \(W_{E/ \mathbb{Q}, f}\). Using the facts that
\begin{itemize}
\item \(\tau \sigma \circ j(\tilde{\alpha} )(\sigma^{-1} \tau^{-1}) =  j(\tilde{\alpha} )(1) \cdot \sigma^{-1} \tau^{-1}\) (where \(1\) denotes the identity in \(\Gamma_{E/\mathbb{Q}}\)); and 
\item \(\tau \sigma \circ j(\tilde{\alpha} ) = [\tilde{\tau}](\sigma \circ j(\tilde{\alpha})) \) 
\end{itemize}
we see this can be expressed as 
\( (\tilde{\tau}^{-1} \cdot w^{-1}_{ j(\tilde{\alpha} )(1) \cdot \sigma^{-1} \tau^{-1}}  \cdot \sigma \circ j(\tilde{\alpha} ) (w_{\sigma^{-1} \tau^{-1}} \tilde{\tau}) )_{j, \sigma}\). We define a new set of coset representatives \(w^{\prime}_{\sigma}\) by \(w^{\prime}_{\sigma} = w_{\sigma \tau^{-1}} \tilde{\tau}\) and by Proposition 5.2 we know \(f_{F}^E(\tilde{\alpha} )\) can be computed as the image of 
\[ (h^{\prime}_{\sigma^{-1}} (\sigma \circ j(\tilde{\alpha})))_{j, \sigma} = (w^{\prime \; -1}_{\sigma \circ j(\tilde{\alpha} )(\sigma^{-1})} \cdot \sigma \circ j(\tilde{\alpha} )(w^{\prime}_{\sigma^{-1}}))_{j, \sigma}\] which can be expressed as  \( (\tilde{\tau}^{-1} \cdot w^{-1}_{j(\tilde{\alpha} )(1) \cdot\sigma^{-1}\tau^{-1}} \cdot \sigma \circ j(\tilde{\alpha} )({w}_{\sigma^{-1} \tau^{-1}} \tilde{\tau}))_{j, \sigma} \).  Hence \( \tau^{-1}(f_{F}^E(\tilde{\alpha} )) = f_{F}^E(\tilde{\alpha} )\) as required. 
\end{proof}
\end{proposition} 

We use a topological argument to conclude that the map \( f_{F}^{E} \) factors through \(\Gamma_{E^{ab}/\mathbb{Q} } \# \Gamma_{E^{ab}/F } \). The group \( W_{E/\mathbb{Q} , f} \# W_{E/F , f}\) inherits a naturally topology coming from the wreath product structure for which the map \(f_{F}^E\) is continuous (this follows from our remarks right at the end of \S2). It follows from Chevalley's theorem on units
 that the induced topology on \(S^E(\mathbb{Q}) \subset  S^E(\mathbb{A}_{\mathbb{Q}, f})\) is the discrete topology (see \cite{Milne-Shih}, \S1.2). As a consequence, \(S^E(\mathbb{Q})\) is closed in \(S^E(\mathbb{A}_{\mathbb{Q}, f})\) and the quotient \(S^E(\mathbb{A}_{\mathbb{Q}, f}) / S^E(\mathbb{Q})\) is totally disconnected. Similarly, if \(E\) is Galois over \(\mathbb{Q}\), then the space \( \left[ S^E(\mathbb{A}_{E,f}) / S^E(E) \right]^{\Gal(E/\mathbb{Q})}  \) is totally disconnected, and this plays an important role in showing that Langlands' Taniyama element is well-defined (see \cite{Fargues}, \S9.9). 
 
It immediately follows that the product group \( \left[ S_F^E(\mathbb{A}_{E,f}) / S_F^E(E) \right]^{\Gal(E/\mathbb{Q})}\) is also totally disconnected, and so it suffices to establish that the kernel of \(W_{E/\mathbb{Q} , f} \# W_{E/F , f} \twoheadrightarrow  \Gamma_{E^{ab}/\mathbb{Q} } \# \Gamma_{E^{ab}/F } \) is connected. If we fix a collection \(\tilde{s} = (\tilde{s}_j)_{j \in \Sigma_F}\) of coset representatives for \(W_{E/F , f}\) in  \(  W_{E/\mathbb{Q} , f}\)  and let \(s = (s_j)_{j\in \Sigma_F} = (\varphi_f(\tilde{s}_j))_{j\in \Sigma_F}\) be the induced collection of coset representatives for \(\Gamma_{E^{ab}/F }\) in  \(  \Gamma_{E^{ab}/\mathbb{Q}}\) then the maps \(\rho_s\) and \(\rho_{\tilde{s}}\) fit into a commutative diagram 
\[ \begin{CD} W_{E/\mathbb{Q} , f} \# W_{E/F , f} @>>>  \Gamma_{E^{ab}/\mathbb{Q} } \# \Gamma_{E^{ab}/F }\\
 @V \rho_{\tilde{s}}V \sim V  @V\rho_sV \sim V \\
  S_{\Sigma_F} \ltimes W_{E/F , f}^{\Sigma_F} @>>>  S_{\Sigma_F} \ltimes \Gamma_{E^{ab}/F } ^{\Sigma_F} \end{CD} \]
 where the bottom line is induced by \(\varphi_f\) on each component. It is clear that the kernel of the bottom line is connected with respect to the natural topology on \(S_{\Sigma_F} \ltimes W_{E/F , f}^{\Sigma_F}\) and so by definition so is the kernel of the top line. We conclude that the map \( f_{F}^{E} \) does indeed factor through the surjection \(W_{E/\mathbb{Q} , f} \# W_{E/F , f} \twoheadrightarrow \Gamma_{E^{ab}/\mathbb{Q} } \# \Gamma_{E^{ab}/F } \). Finally we can show: 

\begin{proposition} For \(\alpha \in  \Gamma_{E^{ab}/\mathbb{Q} } \# \Gamma_{E^{ab}/F }\) the  element \(f_{F}^E(\alpha) \in  \mathcal{S}_F^E(\mathbb{A}_{E,f} ) / \mathcal{S}_F^E(E) \) does not depend on the choice of \(i_v\). 
\begin{proof} Let \(i_v\), \(i_v^{\prime}\) be two continuous maps \(W_{\mathbb{R}} \rightarrow W_{\mathbb{Q}} \) which corresponding to the place \(v\) of \(\overline{\mathbb{Q}}\) coming from the fixed embedding into \(\mathbb{C}\). Then \(i_v\) and \(i_v^{\prime}\) differ by an inner automorphism of \(W_{\mathbb{Q}}\) by an element of \(\Ker(\varphi)\) , and so if \(\tilde{c}\) and \(\tilde{c}^{\prime}\) the associated lifts of complex conjugation to \(W_{E/\mathbb{Q},f }\) we know that we have \(\tilde{c}^{\prime} = u \tilde{c} u^{-1} \) for some \(u \in \Ker(W_{E/\mathbb{Q},f } \xrightarrow{\varphi_f} \Gamma_{E^{ab}/ \mathbb{Q}} )\).
Let  \(f_{F, \tilde{c}}^E\) and \(f_{F, \tilde{c}^{\prime}}^E\) the maps constructed using \(\tilde{c}\) and \(\tilde{c}^{\prime}\) respectively. By definition \(f_{F, \tilde{c}}^E(\alpha)\) is equal to the image of 
\( (w_{j(\tilde{\alpha})\sigma}^{-1} \cdot j(\tilde{\alpha})(w_{\sigma}))_{j \otimes \sigma^{-1}} \)
where \(\tilde{\alpha} \in W_{E/\mathbb{Q}, f } \# W_{E/F,f }\) is any lift of \(\alpha\)  and the \(w_{\sigma}\) are any coset representatives such that \(w_{c\sigma} = \tilde{c} w_{\sigma}\). 

If we set \(w_{\sigma}^{\prime} = uw_{\sigma} \) then \( w^{\prime}_{c\sigma} = \tilde{c}^{\prime} w^{\prime}_{\sigma}\) and \(\tilde{\alpha}^{\prime} \in W_{E/\mathbb{Q}, f } \# W_{E/F,f }\) to be the map defined by \(\tilde{\alpha}^{\prime}(x) = u \tilde{\alpha}( u ^{-1}x) \) then \(\tilde{\alpha}^{\prime}\) is a lift of \(\alpha\) such that 
\[ (w^{\prime \; -1}_{j(\tilde{\alpha}^{\prime})\sigma} \cdot j(\tilde{\alpha}^{\prime})(w^{\prime}_{\sigma}))_{j \otimes \sigma^{-1}} =  (w_{j(\tilde{\alpha})\sigma}^{-1}\cdot u^{-1} j(\tilde{\alpha}^{\prime})(u w_{\sigma}))_{j \otimes \sigma^{-1}} =  (w_{j(\tilde{\alpha})\sigma}^{-1} \cdot  j(\tilde{\alpha})(w_{\sigma}))_{j \otimes \sigma^{-1}}\]
as one can check from the definitions that \(j(\tilde{\alpha}^{\prime}) = j(\tilde{\alpha})^{\prime} \). 
Hence \(f_{F, \tilde{c}}^E(\alpha) =  f_{F, \tilde{c}^{\prime}}^E(\alpha)\) as required. 
\end{proof}
\end{proposition}
This competes the proof of Theorem 5.1.
\end{proof} 
\end{theorem}

We end this section by studying the properties of the \(F\)-plectic universal Taniyama element \(f_F^E\) as we vary the fields \(F\) and \(E\). 

\begin{proposition} Let \(F \subset \overline{\mathbb{Q}}\) be a totally real number field and \(E \subset \overline{\mathbb{Q}}\) be CM, Galois over \(\mathbb{Q}\) and containing \(F\). Then:

(i)  If \(F \subset F^{\prime} \subset E\) is totally real then the maps \(f_{F}^E \) and \(f_{F^{\prime}}^E\) fit into a commutative diagram
\[ \begin{CD} \Aut_F(F \otimes_{\mathbb{Q}} E^{ab})  @> f_F^E >> \mathcal{S}_F^{E}(\mathbb{A}_{{E},f}) /  \mathcal{S}_F^{E}(\mathbb{A}_{{E},f}) \\
@VVV @V \text{ diag }VV \\
\Aut_{F^{\prime}}(F^{\prime} \otimes_{\mathbb{Q}} E^{ab})   @> f_{F^{\prime}}^E >>   \mathcal{S}_{F^{\prime}}^E(\mathbb{A}_{{E},f}) /  \mathcal{S}_{F^{\prime}}^E(\mathbb{A}_{E,f}) \end{CD} \]
where \(\text{diag}\) denotes the diagonal morphism \(\mathcal{S}^E_{F} \rightarrow \mathcal{S}_{F^{\prime}}^E\) induced by the natural surjection \(\Sigma_{F^{\prime}} \twoheadrightarrow \Sigma_{F}\).  

(ii)  Let \(u \in \Gamma_{\mathbb{Q}} \) and \(F^{\prime} = u(F) \subset \overline{\mathbb{Q}}\). The maps \(f_F^E\) and \(f_{F^{\prime}}^E\) fit into a commutative diagram\footnote{Throughout this paper we have taken our totally real field \(F\) to be an embedded subfield of \(\overline{\mathbb{Q}}\). Although this simplifies the exposition, it is not necessary. Indeed, by Proposition 5.5(ii), if \(F\) is any abstract totally real field and \(E \subset \overline{\mathbb{Q}}\) is CM, Galois over \(\mathbb{Q}\) and splits \(F\) then the composite 
\[f_F^E : \Aut_F(F\otimes E^{ab} ) \xrightarrow{j}  \Aut_{j(F)}(j(F)\otimes E^{ab} ) \xrightarrow{f_{j(F)}^E} \mathcal{S}_{j(F)}^{E}(\mathbb{A}_{{E},f}) /  \mathcal{S}_{j(F)}^{E}(\mathbb{A}_{{E},f}) \xrightarrow{j^{-1}} \mathcal{S}_F^E(\mathbb{A}_{E,f})/\mathcal{S}_F^E({E})\] is independent of the choice of embedding \(j : F \hookrightarrow E\) and so defines a canonical \(F\)-plectic universal Taniyama element. }:
\[ \begin{CD}\Aut_F(F \otimes_{\mathbb{Q}} E^{ab}) @> f_F^E >>  \mathcal{S}_F^{E}(\mathbb{A}_{{E},f}) /  \mathcal{S}_F^{E}(\mathbb{A}_{{E},f})  \\
@V [u] VV @VV u_{\ast} V \\
\Aut_{F^{\prime}}(F^{\prime} \otimes_{\mathbb{Q}} E^{ab})  @> f_{F^{\prime}}^E>> \mathcal{S}_{F^{\prime}}^{E}(\mathbb{A}_{{E},f}) /  \mathcal{S}_{F^{\prime}}^{E}(\mathbb{A}_{{E},f})   \end{CD} \] 
where the map \(u_{\ast}\) on the right hand side is induced by the bijection \(\Sigma_{F^{\prime}} \rightarrow \Sigma_{F}\) given by \( j^{\prime} \mapsto j = j^{\prime} \circ u|_F\).

\begin{proof} The inclusion \( \Aut_F(F \otimes_{\mathbb{Q}} E^{ab}) \hookrightarrow \Aut_{F^{\prime}}(F^{\prime} \otimes_{\mathbb{Q}} E^{ab}) \) corresponds to the natural inclusion of plectic groups. Let  \(\alpha \) be an element of \( \Gamma_{E^{ab}/\mathbb{Q}} \# \Gamma_{E^{ab}/F}  \subset \Gamma_{E^{ab}/\mathbb{Q}} \# \Gamma_{E^{ab}/F^{\prime}}  \) and \(\tilde{\alpha}\) be a lift of \(\alpha\) to an element of \(W_{E /\mathbb{Q}, f}\# W_{E / F, f}\), which can view as a subgroup of \(W_{E /\mathbb{Q}, f}\# W_{E / F^{\prime}, f} \). By definition we have \(f^E_{F}(\alpha)\) equal to the image of 
\((h_{\sigma^{-1}}(  j(\tilde{\alpha})))_{j \otimes \sigma} \)
under \(T_{F \otimes E} \twoheadrightarrow \mathcal{S}_F^E \). Similarly \(f^E_{F^{\prime}}(\alpha)\) is the image of 
\((h_{\sigma^{-1}}(  j^{\prime}(\tilde{\alpha})))_{j^{\prime} \otimes \sigma} \)
under \(T_{F^{\prime} \otimes E} \twoheadrightarrow \mathcal{S}_{F^{\prime}}^E \) and, as \(\tilde{\alpha}\) lies in the subgroup \(W_{E /\mathbb{Q}, f}\# W_{E / F^{\prime}, f}\), the image of \(\tilde{\alpha}\) under \(j^{\prime}\) only depends on \(j^{\prime}|_F\). Hence \(f^E_{F^{\prime}}(\alpha)\) is equal to the image of 
\( (h_{\sigma^{-1}}( j^{\prime}|_F(\tilde{\alpha})))_{j^{\prime} \otimes \sigma} \)
under \(T_{F \otimes E} \twoheadrightarrow \mathcal{S}_F^E \) which is precisely \(\text{ diag }(f_{F}^E(\alpha))\). 
 
 To prove (ii) recall that for \(\alpha \in \Gamma_{E^{ab}/\mathbb{Q}} \# \Gamma_{E^{ab}/F} \) by definition \(f^E_F(\alpha)\) is the image of 
\( (h_{\sigma^{-1}}(j(\tilde{\alpha})))_{j \otimes \sigma} \)
where \(j(\tilde{\alpha})\) means \( [\tilde{v}] \tilde{\alpha} \) and \(\tilde{v} \in W_{E /\mathbb{Q}, f}\) is such that \( v |_{F} = j \).  Similarly for \(\alpha^{\prime} \in \Gamma_{E^{ab}/\mathbb{Q}} \# \Gamma_{E^{ab}/F^{\prime}} \) then  \(f^E_{F^{\prime}}( \alpha^{\prime})\) is the image of 
\( (h_{\sigma^{-1}}( j^{\prime}(\tilde{\alpha}^{\prime})))_{j^{\prime} \otimes \sigma} \) where
\(j^{\prime}(\tilde{\alpha}^{\prime})\) means \( [\tilde{w}] \tilde{\alpha}^{\prime} \) and \(\tilde{w} \in W_{E /\mathbb{Q}, f}\) is such that \(w |_{F^{\prime}} = j^{\prime} \). The map \(u_{\ast}\) on the right hand side of the diagram sends \([j^{\prime} \otimes \sigma] \mapsto [j \otimes \sigma]\) where \(j = j^{\prime} \circ u|_{F}\).  The map \([u] : \Aut_F(F \otimes_{\mathbb{Q}} E^{ab}) \rightarrow \Aut_{F^{\prime}}(F^{\prime} \otimes_{\mathbb{Q}} E^{ab}) \) on the left hand side is induced by \([\tilde{u}] \) on plectic groups, and the \( j^{\prime} \otimes \sigma \) component of \(f^E_{F^{\prime}}([\tilde{u}] \tilde{\alpha})\) can be computed as 
\[ h_{\sigma^{-1}}( j^{\prime}([\tilde{u}]\tilde{\alpha})) = h_{\sigma^{-1}}([ \tilde{w}][\tilde{u}] \tilde{\alpha}) =  h_{\sigma^{-1}}([ \tilde{w}\tilde{u}] \tilde{\alpha}) = h_{\sigma^{-1}}(j(\tilde{\alpha}))\]
as \({w}{u} |_F = j\) if \(w|_{F^{\prime}} = j^{\prime} = j \circ u|_{F}^{-1}\). 
\end{proof}
\end{proposition}

\begin{proposition}  Fix \(F\) and suppose \(E^{\prime} \supset E\) is also CM and Galois over \(\mathbb{Q}\). There is a commutative diagram
\[ \begin{CD} \Aut_F(F \otimes_{\mathbb{Q}} E^{\prime  ab}) @> f_F^{E^{\prime}} >> \mathcal{S}_F^{E^{\prime}}(\mathbb{A}_{E^{\prime} ,f}) /  \mathcal{S}_F^{E^{\prime}}(E^{\prime}) \\
@VVV@V V N_{E^{\prime} / E} V \\
\Aut_F(F \otimes_{\mathbb{Q}} E^{ab}) @> f_F^{E} >>   \mathcal{S}_F^E(\mathbb{A}_{E^{\prime},f}) /  \mathcal{S}_F^E(E^{\prime}) \end{CD} \]
where here we are viewing \(\mathcal{S}_F^E(\mathbb{A}_{{E},f}) /  \mathcal{S}_F^E(E)\) as a subset of \( \mathcal{S}_F^E(\mathbb{A}_{{E^{\prime}},f}) /  \mathcal{S}_F^E(E^{\prime})\)

\begin{proof} In terms of plectic groups, we need to show that we have a commutative diagram  \[ \begin{CD} \Gamma_{E^{\prime ab}/\mathbb{Q}} \# \Gamma_{E^{\prime ab} /F }  @> f_F^{E^{\prime}} >> \mathcal{S}_F^{E^{\prime}}(\mathbb{A}_{{E^{\prime}},f}) /  \mathcal{S}_F^{E^{\prime}}(E^{\prime}) \\
@VVV@V V N_{E^{\prime} / E} V \\
\Gamma_{E^{ab}/\mathbb{Q}} \# \Gamma_{E^{ab}/F }  @> f_F^{E} >>   \mathcal{S}_F^E(\mathbb{A}_{{E^{\prime}},f}) /  \mathcal{S}_F^E(E^{\prime}) \end{CD} \]
where the left hand map is the natural surjection induced by \(\Gamma_{E^{\prime ab} / \mathbb{Q}} \twoheadrightarrow \Gamma_{E^{ab} / \mathbb{Q}} \). 

To this end let \(W_{\mathbb{Q}} \supset  W_{E} \supset W_{E^{\prime}} \) be the relevant Weil groups and let \(W_{E/\mathbb{Q}} = W_{\mathbb{Q}} / [W_E, W_E] \) and \(W_{E^{\prime}/\mathbb{Q}} = W_{\mathbb{Q}} / [W_{E^{\prime}}, W_{E^{\prime}}] \). We set \(G\) to be \(W_{{E^{\prime}}/\mathbb{Q}, f}\),  \( H \) to be \( W_{E^{\prime} /E, f} \) and \(H^{\prime} = W_{E^{\prime} / E^{\prime}, f} = \mathbb{A}_{{E^{\prime}},f}^{\times} / {E^{\prime}}^{\times}\), and can identify \( W_{E/\mathbb{Q}, f} \) with \(G / [H,H]\). 
We are in the setting of Proposition 2.3 and in that notation we have \(X = G /H = \Gamma_{E/ \mathbb{Q}}\) and \( X^{\prime} = G / H^{\prime} = \Gamma_{E^{\prime} / \mathbb{Q}}\) and we set \(Y = H/H^{\prime}\). The transfer map 
\[ \Ver_{E^{\prime}/ E} :  H^{ab} = \mathbb{A}_{E,f}^{\times} / E^{\times} \rightarrow H^{\prime \; ab} = H = \mathbb{A}_{E^{\prime}, f}^{\times} / E^{\prime \; \times}  \] 
associated to the inclusion \(H^{\prime} \subset H\) is just the obvious inclusion \(\mathbb{A}_{E,f}^{\times} / E^{\times} \hookrightarrow \mathbb{A}_{E^{\prime}, f}^{\times} / E^{\prime \; \times} \) (see \cite{Tate}, \S1). 

Our fixed map \(i_v : W_{\mathbb{C} / \mathbb{R}} \rightarrow W_{\mathbb{Q}}\) gives us an element \(\tilde{c} \in W_{E/\mathbb{Q}, f} \) lifting complex conjugation \(c \in  \Gamma_{E^{\prime}/\mathbb{Q}} \) and whose image in \(W_{E/\mathbb{Q}}\) (which we'll denote \(\overline{\tilde{c}}\)) does likewise for \(E\). If we pick coset representatives \(w_{\sigma}\)  for \(H\) in  \(G\) such that for each \(\sigma \in \Gamma_{E / \mathbb{Q}}\) we have \(w_{c \sigma} = \tilde{c}w_{\sigma}\) and any set \(s_y\) for \(y \in Y\) of coset representatives for \(H^{\prime} \subset H\) then

\begin{itemize} 
\item The set \(\{w_{\sigma^{\prime}}\} = \{w_{\sigma}s_y : \sigma \in \Gamma_{E / \mathbb{Q}} , y \in Y\}\) give a set of coset representatives for \(H^{\prime}\) in \(G\) satisfying the relevant criterion for \(\tilde{c}\) (indeed \(w_{c \sigma}s_y = \tilde{c}w_{\sigma}s_y\)). 
\item The the images  \(\overline{w}_{\sigma}\) in \(W_{E/\mathbb{Q}, f} = G / [H,H]\) give a set of coset representatives for \(W_{E/E, f} = \mathbb{A}_{E,f}^{\times} / E^{\times}\) in \(W_{E/\mathbb{Q}, f} \) such that \(\overline{w}_{c \sigma} = \overline{\tilde{c}} \overline{w}_{\sigma}\). 
\end{itemize}

By definition for \( \Gamma_{E^{\prime ab}/\mathbb{Q}} \# \Gamma_{E^{\prime ab} /F }\) we have \(f_F^{E^{\prime}}(\alpha)\) equal to the image of 
\( ( h_{\sigma^{\prime}}(j (\tilde{\alpha})))_{j \otimes {\sigma^{\prime}}^{-1}} \in \prod_{j \otimes \sigma^{\prime}} \mathbb{A}_{E^{\prime},f}^{\times} / E^{\prime \times} \)
under \(T_{F \otimes E} \twoheadrightarrow \mathcal{S}_F^{E^{\prime}}\).  The map \(N_{E^{\prime}/ E}\) sends 
\( [j \otimes \sigma] \mapsto \sum_{\sigma^{\prime}|_E = \sigma} [j \otimes \sigma^{\prime}]\) and hence by Proposition 2.3 (we're using that \(j(\tilde{\alpha}) \in G \# H \subset G \# H^{\prime}\) as \( j(F) \subset E \subset E^{\prime}\)) the \((j \otimes \sigma^{-1})\) component of \(N_{E^{\prime}/ E} \circ f_F^{E^{\prime}}(\alpha)\) is equal to 
\[ \prod_{ \sigma^{\prime}|_E = \sigma} h_{\sigma^{\prime}}( j (\tilde{\alpha}))  = \prod_{ \sigma^{\prime}|_E = \sigma} h_{\sigma^{\prime}}(j (\tilde{\alpha})) =   \prod_{ y \in Y} s_{h_{j, \sigma} y}^{-1} h_{j, \sigma} s_y \in \mathbb{A}_{E^{\prime}, f}^{\times} / E^{\prime \times} \]
where \(h_{j, \sigma} = h_{\sigma}(j (\tilde{\alpha})) \in H = W_{E^{\prime} / E, f } \). But this  right hand side is precisely \( \Ver_{E^{\prime} /E} \overline{h_{j, \sigma}} \), where \(\overline{h_{j, \sigma}}\) is the image of \(h_{j, \sigma}\) in \(H^{ab} = \mathbb{A}_{E, f}^{\times} /E^{\times} \) and we have already remarked that \(\Ver_{E^{\prime} /E}\) is identified with the inclusion \( \mathbb{A}_{E, f}^{\times} /E^{\times} \hookrightarrow \mathbb{A}_{E^{\prime}, f}^{\times} /E^{\prime  \times} \). The result then follows since if we let \(\overline{\alpha} \in  \Gamma_{E^{ab}/\mathbb{Q}} \# \Gamma_{E^{ab} /F }  \) be the reduction of \(\alpha\) then the \((j \otimes \sigma^{-1})\) component of \( f_F^{E}(\overline{\alpha})\) (computed with the coset representatives \(\overline{w}_{\sigma}\)) is given by 
\( \overline{h_{j, \sigma}} = \overline{h_{\sigma}( j(\tilde{\alpha}))} \in \mathbb{A}_{E, f}^{\times} /E^{\times} \).  
\end{proof}

\end{proposition}

\section{A plectic Taniyama group}

The data of Langlands' universal Taniyama element \(f^E : \Gamma_{E^{ab} / \mathbb{Q}} \rightarrow \mathcal{S}^E(\mathbb{A}_{E,f}) / \mathcal{S}^E(E)\) (for \(E \subset \overline{\mathbb{Q}}\) a CM number field which is Galois over \(\mathbb{Q}\)) is equivalent to the data of the level \(E\) Taniyama group
\[ 1 \rightarrow \mathcal{S}^E \rightarrow \mathcal{T}^E \rightarrow \Gamma_{E^{ab} / \mathbb{Q}} \rightarrow 1\]
together with its continuous finite-adelic splitting \(s^E :  \Gamma_{E^{ab} / \mathbb{Q}} \rightarrow \mathcal{T}^E(\mathbb{A}_{\mathbb{Q}, f})\). This is a special case of the following more general formalism: 

\begin{proposition} Let \(S\) be any torus split over a Galois extension \(E/\mathbb{Q}\) and \(\Gamma\) a profinite group. Suppose we are given an `algebraic' action \(\Gamma \rightarrow \Aut(S)\) which factors through a finite quotient of \(\Gamma\) and an extension
\[1 \rightarrow S \rightarrow T \rightarrow \Gamma \rightarrow 1 \]
inducing the given algebraic action of \(\Gamma\) on \(S\) together with a continuous finite-adelic splitting \(s : \Gamma \rightarrow T(\mathbb{A}_{\mathbb{Q},f})\). Pick a section \(a : \Gamma \rightarrow T_{/E}\) that is a morphism of pro-algebraic varieties (this exists as \({S}\) is split over \(E\)), and define \[ {b} (\gamma) =  s(\gamma)a(\gamma)^{-1}  \in S(\mathbb{A}_{E,f} )\]
 for \(\gamma\) in \(\Gamma \). The  induced map 
 \[ \overline{b} : \Gamma \rightarrow S(\mathbb{A}_{E,f} ) /{S}(E)\] i.e. \( \gamma \mapsto s(\gamma) a(\gamma)^{-1} \text{ mod } {S}(E)\), is independent of the choice of \(a\) and satisfies the following three properties:
\begin{enumerate}
\item \( \overline{b} \) is a cocycle for the algebraic action of \( \Gamma \) on \(S(\mathbb{A}_{E,f} ) / S(E)\), i.e. for all \(\gamma_1, \gamma_2 \in \Gamma\) we have
\[ \overline{b}(\gamma_1 \gamma_2) =  \overline{b}(\gamma_1) \cdot  {\gamma_1} \star \overline{b}(\gamma_2) \]  
{where \(\star\) denotes the given algebraic action}
\item  \(\overline{b}\) lands in the Galois invariant part \( \left[ S(\mathbb{A}_{E,f} ) / S(E) \right] ^{\Gal(E/\mathbb{Q})} \subset S(\mathbb{A}_{E,f} ) / S(E) \) 
\item  \(\overline{b}\) admits a continuous lift to a map \(b : \Gamma \rightarrow S(\mathbb{A}_{E,f} )\) such that
\[ (\gamma_1, \gamma_2) \mapsto  d_{\gamma_1, \gamma_2} := b(\gamma_1)\cdot {\gamma_1}\star b(\gamma_2) \cdot b(\gamma_{1}\gamma_{2})^{-1} \in S(E) \subset S(\mathbb{A}_{f,E})  \]
 is locally constant
\end{enumerate}
Conversely any map  \( \overline{b} : \Gamma  \rightarrow {S}(\mathbb{A}_{E,f} ) /{S}(E)\)  satisfying (1), (2) and (3) arises from a unique extension of \(\Gamma\) by \(\mathcal{S}\) (together with a continuous finite-adelic splitting). 

\begin{proof} This is proven in \cite{Milne-Shih} Proposition 2.7 for \(\Gamma = \Gamma_{E^{ab} / \mathbb{Q}}\) acting on \(S= \mathcal{S}^E\) by the algebraic action defined in \S4. However, as noticed by Fargues \cite{Fargues} and Nekov\'a\v{r} \cite{Nekovar}, the argument is not specific to that case. 
\end{proof}
\end{proposition}

In the above setting, if \(S^{\prime}\) is another torus split by \(E\) and \(S \rightarrow S^{\prime}\) is any homomorphism then the extension defined by the composite map 
\( \Gamma \xrightarrow{\overline{b}} {S}(\mathbb{A}_{E,f} ) /{S}(E)  \xrightarrow{} {S^{\prime}}(\mathbb{A}_{E,f} ) /{S^{\prime}}(E) \)
is just the push forward of \(T\) via \(S \rightarrow S^{\prime}\).  Similarly the pull back of \(T\) via a homomorphism  \(\Gamma^{\prime} \rightarrow \Gamma\)  corresponds to the composite 
\( \Gamma^{\prime} \rightarrow \Gamma \xrightarrow{\overline{b}} {S}(\mathbb{A}_{E,f} ) /{S}(E)  \) (see \cite{Fargues}, \S9).

\begin{example} Let \(E \subset \overline{\mathbb{Q}}\) be a CM field which is Galois over \(\mathbb{Q}\). Under Proposition 6.1, Langlands' level \(E\) Taniyama extension 
\[ 1 \rightarrow \mathcal{S}^E \rightarrow \mathcal{T}^E \rightarrow \Gamma_{E^{ab} / \mathbb{Q}} \rightarrow 1\]
corresponds to the map \(\overline{b}^E : g \mapsto f^E(g^{-1})^{-1}\), which can be checked to satisfy the properties (1), (2) and (3). Previously, Serre \cite{Serre} had constructed an extension
\[ 1 \rightarrow \mathcal{S}^E \rightarrow \epsilon^E \rightarrow \Gamma_{E}^{ab} \rightarrow 1\] 
together with a continuous finite-adelic splitting 
which was related to the theory of algebraic Hecke characters. Under Proposition 6.1 this corresponds to the homomorphism \[N^E:  \Gamma_{E}^{ab} \xrightarrow{art_E^{-1}} \mathbb{A}_{E,f}^{\times} / \overline{E^{\times}} \xrightarrow{\mu^E} \mathcal{S}^E(\mathbb{A}_{E, f}) / \overline{\mathcal{S}^E(E)}\xrightarrow{N_{E/\mathbb{Q}}} \mathcal{S}^E(\mathbb{A}_{\mathbb{Q}, f}) / \mathcal{S}^E({\mathbb{Q}}) \] 
One can check from the definition that \(f^E|_{\Gamma_E^{ab}} = N^E\), i.e. the Serre extension is the pull back of the Taniyama extension via \(\Gamma_{E}^{ab}  \hookrightarrow \Gamma_{E^{ab} /\mathbb{Q}}\). Similarly, if \(F \subset E\) is totally real and \({g} = (g_j)_{j \in \Sigma_F} \)  lies in the subgroup \(\Gamma_{E^{ab}/E}^{\Sigma_F}  \subset \Aut_F(F\otimes_{\mathbb{Q}} E^{ab})\) then \( \alpha_{{g}} : \Gamma_{E^{ab}/\mathbb{Q}} \rightarrow \Gamma_{E^{ab}/\mathbb{Q}} \) is the bijection given by left multiplication by \(g_j\) on \(\{ g \in \Gamma_{E^{ab}/\mathbb{Q}} : g|_F = j \}\). If we let \(N_F^E\) denote the componentwise map 
\[ N_F^E = (N^E, \dots, N^E):  \Gamma_{E^{ab}/E}^{\Sigma_F} \rightarrow \mathcal{S}_F^E(\mathbb{A}_{\mathbb{Q}, f}) / \mathcal{S}_F^E(\mathbb{Q})\]
one can check straight from the definition that if \(g \in \Gamma_{E^{ab}/E}^{\Sigma_F}\) then \(f_F^E(g) = N_F^E(g)\). 
\end{example}

From now on we place ourselves in the setting of \S5.  Fix a total real field \(F \subset \overline{\mathbb{Q}}\) and let \(E \subset \overline{\mathbb{Q}}\) be a CM field which is Galois over \(\mathbb{Q}\) and contains \(F\). In \S4 we defined an algebraic action of \( \Aut_F(F\otimes_{\mathbb{Q}} E^{ab}) \) on \(\mathcal{S}_F^E\) via the quotient \(\Aut_F(F\otimes_{\mathbb{Q}} E^{ab}) \twoheadrightarrow \Aut_F(F\otimes_{\mathbb{Q}} E)\) and in \S5 we constructed an \(F\)-plectic analogue 
\[ f_F^E : \Aut_F(F\otimes_{\mathbb{Q}} E^{ab})  \rightarrow \mathcal{S}_F^E(\mathbb{A}_{E,f})/S_F^E(E) \]
of Langlands' universal Taniyama element. 
The main theorem of this section asserts that this map defines an \(F\)-plectic analogue of the level \(E\) Taniyama extension:

\begin{theorem} There exists a unique extension \[ 1 \rightarrow \mathcal{S}_F^E \rightarrow \mathcal{T}_F^E \rightarrow  \Aut_F(F\otimes E^{ab})  \rightarrow 1\]
together with a continuous finite-adelic splitting \(s^E_F :  \Aut_F(F\otimes E^{ab}) \rightarrow \mathcal{T}_F^E(\mathbb{A}_{\mathbb{Q}, f}) \) giving rise to the map \(\overline{b}_F^E : \Aut_F(F\otimes_{\mathbb{Q}} E^{ab})  \rightarrow \mathcal{S}_F^E(\mathbb{A}_{E,f})/S_F^E(E) \) defined by \(g \mapsto f_F^E(g^{-1})^{-1}\). \end{theorem} 


\begin{proof}
In order to prove Theorem 6.2 we need to establish that \(\overline{b}_F^E\) satisfies properties (1), (2) and (3) of Proposition 6.1. The following proposition handles property (1): 
\begin{proposition}  For \(g_1, g_2 \in \Aut_F(F\otimes_{\mathbb{Q}} E^{ab}) \) we have 
\[ f_F^E(g_1 g_2) = \;   g_2^{-1} \star f_F^E(g_1) \cdot f_F^E(g_2) \]  
\begin{proof} Let \(\alpha_1\) and \(\alpha_2\) be the images of \(g_1\) and \(g_2\) under the canonical isomorphism \(\Aut_F(F\otimes_{\mathbb{Q}} E^{ab} ) \xrightarrow{\sim} \Gamma_{E^{ab}/\mathbb{Q} } \# \Gamma_{E^{ab}/F } \). The algebraic action of any \(\alpha \in \Gamma_{E^{ab}/\mathbb{Q} } \# \Gamma_{E^{ab}/F }\) on the product 
\(\prod_{ j \otimes \sigma }  \mathbb{A}_{E,f}^{\times} / E^{\times} \)
maps \( (e_{j \otimes \sigma})_{j\otimes\sigma} \mapsto (e_{j,  \sigma j(\alpha)})_{j \otimes \sigma}\) where \(\sigma j(\alpha) = (j(\alpha)^{-1} (\sigma^{-1}))^{-1}\). We deduce that \( \alpha_2^{-1} \star f_F^E(\alpha_1) \) is the image of 
\( ( w_{j(\alpha_1) j(\alpha_2) \sigma^{-1}}^{-1} \cdot j(\tilde{\alpha}_1)(w_{j(\alpha_2)\sigma^{-1}}))_{j \otimes \sigma} \)
under \(T_{F \otimes E}  \twoheadrightarrow \mathcal{S}_F^E\). Hence \( {\alpha_2^{-1}} \star f_F^E(\alpha_1)  \cdot f_F^E(\alpha_2) \) is the image of 
\[   ( w_{j(\alpha_1) j(\alpha_2) \sigma^{-1}}^{-1} \cdot j(\tilde{\alpha}_1)(w_{j(\alpha_2)\sigma^{-1}}) \cdot w^{-1}_{j(\alpha_2) \sigma^{-1}}  \cdot j(\tilde{\alpha}_2)(w_{\sigma^{-1}}))_{j \otimes \sigma}\] 
under \(T_{F \otimes E}  \twoheadrightarrow \mathcal{S}_F^E\). But equally we can write
\( j(\tilde{\alpha}_1 \tilde{\alpha}_2)(w_{\sigma^{-1}}) \) as \( j(\tilde{\alpha}_1) (w_{j(\alpha_2) \sigma^{-1}} \cdot w_{j(\alpha_2) \sigma^{-1}}^{-1}  \cdot j(\tilde{\alpha}_2)(w_{\sigma^{-1}}) ) \) which equals  \( j(\tilde{\alpha}_1)( w_{j(\alpha_2) \sigma^{-1}}) \cdot w_{j(\alpha_2) \sigma^{-1}}^{-1} \cdot j(\tilde{\alpha}_2)(w_{\sigma^{-1}}) \) as \(w_{j(\alpha_2) \sigma^{-1}}^{-1} j(\tilde{\alpha}_2)(w_{\sigma^{-1}}) \in  \mathbb{A}_{E,f}^{\times}/ E^{\times}\). Hence \(f_F^E(\alpha_1 \alpha_2)\) is also the image of  
\[  (w^{-1}_{j(\alpha_1 \alpha_2) \sigma^{-1}} \cdot j(\tilde{\alpha}_1)( w_{j(\alpha_2) \sigma^{-1}}) \cdot w_{j(\alpha_2) \sigma^{-1}}^{-1} \cdot j(\tilde{\alpha}_2)(w_{\sigma^{-1}}))_{j \otimes \sigma} \] 
under \(T_{F \otimes E}  \twoheadrightarrow \mathcal{S}_F^E\) and we conclude that \(f_F^E(\alpha_1 \alpha_2) = \;   \alpha_2^{-1} \star f_F^E(\alpha_1) \cdot f_F^E(\alpha_2) \). \end{proof}
\end{proposition}

It follows immediately from the above proposition that \(\overline{b}_F^E\) satisfies property (1) of Proposition 6.1. On the other hand property (2) follows immediately from Proposition 5.3. Therefore, it only remains to prove:

\begin{proposition} There is a continuous map \(b_F^E :  \Aut_F(F\otimes E^{ab}) \rightarrow  \mathcal{S}_F^E(\mathbb{A}_{{E},f})  \) lifting \(\overline{b}_F^E :   \Aut_F(F\otimes E^{ab}) \rightarrow  \mathcal{S}_F^E(\mathbb{A}_{E,f}) /  \mathcal{S}_F^E(E  )\) such that the map 
\[ (g_1, g_2) \mapsto  d_{g_1, g_2} := b_F^E(g_1) \cdot {g_1} \star b_F^E(g_2) \cdot b_F^E(g_{1}g_{2})^{-1} \in \mathcal{S}_F^E(E) \subset \mathcal{S}_F^E(\mathbb{A}_{E,f}) \] is locally constant. 
\begin{proof} In \cite{Milne-Shih}, it is shown that there is a finite abelian extension \(E^{\prime} / E\) and a continuous homomorphism \(b : \Gamma_{E^{ab}/E^{\prime}} \rightarrow \mathcal{S}^E(\mathbb{A}_{\mathbb{Q},f})  \) fitting in a diagram 
\[ \begin{CD} 1 @>>> \Gamma_{E^{ab}/E^{\prime}} @>>> \Gamma_{E^{ab}/E}@>>> \Gamma_{E^{\prime}/E} @>>> 1 \\
&&  @V b VV @V N^E VV && && \\ 
&&   \mathcal{S}^E(\mathbb{A}_{\mathbb{Q},f})  @>>> \mathcal{S}^E(\mathbb{A}_{\mathbb{Q},f}) /\mathcal{S}^E(\mathbb{Q})  && && \end{CD} \]
The map \((b, \dots, b) : \Gamma_{E^{ab}/E^{\prime}}^{\Sigma_F} \rightarrow \mathcal{S}_F^E(\mathbb{A}_{\mathbb{Q},f}) \subset \mathcal{S}_F^E(\mathbb{A}_{E,f})  \) can be extended to a continuous map \(b_F^E :\Aut_F(F\otimes E^{ab}) \rightarrow  \mathcal{S}_F^E(\mathbb{A}_{E,f})  \) lifting \(\overline{b}_F^E\) by, for example, picking coset representatives (see \cite{Milne-Shih},  \S2.10). This \(b_F^E\) satisfies the condition of the proposition as when restricted to the open subgroup \(\Gamma_{E^{ab}/E^{\prime}}^{\Sigma_F}\) (which acts trivially) it is a homomorphism. 
\end{proof}
\end{proposition}

This completes the proof of Theorem 6.2. \end{proof}

Finally we can construct the full \(F\)-plectic Taniyama group.  It follows from Proposition 5.6 and Proposition 6.1 that whenever \(E \subset E^{\prime}\) we have a transition map \(\mathcal{T}_F^{E^{\prime}} \rightarrow \mathcal{T}_F^E\) which sits in a diagram 
\[  \begin{CD}1 @>>> \mathcal{S}_F^{E^{\prime}} @>>> \mathcal{T}_F^{E^{\prime}} @>>> \Aut_F(F\otimes E^{\prime ab}) @>>> 1 \\
&& @V N_{E^{\prime}/ E} VV @VVV @VVV  &&\\
1 @>>> \mathcal{S}_F^{E} @>>> \mathcal{T}_F^{E} @>>> \Aut_F(F\otimes E^{ab}) @>>> 1 \end{CD} \]
and respects the finite-adelic splittings. We can pull back each extension \(\mathcal{T}_F^E\) by \(\Aut_F(F\otimes_{\mathbb{Q}} \overline{\mathbb{Q}}) \twoheadrightarrow \Aut_F(F\otimes_{\mathbb{Q}} E^{ab})  \), and in the limit we obtain an extension of the form 
\[ 1 \rightarrow \mathcal{S}_F \rightarrow \mathcal{T}_F \rightarrow \Aut_F(F\otimes_{\mathbb{Q}} \overline{\mathbb{Q}}) \rightarrow 1\]
together with a continuous finite-adelic splitting \(s_F : \Aut_F(F\otimes_{\mathbb{Q}} \overline{\mathbb{Q}}) \rightarrow  \mathcal{T}_F(\mathbb{A}_{\mathbb{Q}, f}) \), where here  \(\mathcal{S}_F\) denotes the \(F\)-plectic Serre group \(\mathcal{S}^{\Sigma_F}\). When \(F = \mathbb{Q}\) this is precisely the definition of Langlands' Taniyama group \(\mathcal{T}\). 
Moreover, Proposition 5.5 ensures that the construction is functorial in \(F\). A map \(F \rightarrow F^{\prime}\) of totally real fields induces a homomorphism \(\mathcal{T}_{F} \rightarrow \mathcal{T}_{F^{\prime}}\) restricting to the natural map on Serre groups and inducing the inclusion \(\Aut_F(F\otimes \overline{\mathbb{Q}} ) \hookrightarrow \Aut_{F^{\prime}}(F^{\prime} \otimes \overline{\mathbb{Q}} )  \). In particular, there is a canonical homomorphism from Langlands' Taniyama group  \(\mathcal{T}\) to the the \(F\)-plectic Taniyama group \(\mathcal{T}_F\) for any totally real number field \(F\). \\
\; \\
\textsl{Acknowledgements}: This paper is the result of research begun while studying for a PhD at the University of Cambridge. I am grateful to my supervisor Tony Scholl for his excellent guidance throughout.  I would also particularly like to thank Jan Nekov\'a\v{r} for explaining to me his paper \cite{Nekovar} and for helpful discussions about this work. 
During my PhD I was supported by the Engineering and Physical Sciences Research Council and this paper was completed while I was a Heilbronn Research Fellow at King's College, London. \\

\end{document}